\documentclass[onefignum,onetabnum]{siamart171218}

\usepackage{lipsum}
\usepackage{amsfonts}

\usepackage{amssymb}
\usepackage{amsmath}

\usepackage{graphicx}
\usepackage{epstopdf}
\usepackage{algorithmic}
\ifpdf
  \DeclareGraphicsExtensions{.eps,.pdf,.png,.jpg}
\else
  \DeclareGraphicsExtensions{.eps}
\fi

\newsiamremark{remark}{Remark}
\newsiamremark{hypothesis}{Hypothesis}
\crefname{hypothesis}{Hypothesis}{Hypotheses}
\newsiamthm{claim}{Claim}

\headers{Eigenproblems from families of coupled systems}{A. Hannukainen, J. Malinen, and A. Ojalammi}

\graphicspath{{figures/}}
\usepackage{tikz}
\usetikzlibrary{arrows,calc}

\usepackage[safe]{tipa}

\usepackage{stmaryrd}

\renewcommand{\vec}[1]{ \boldsymbol{#1}}

\renewcommand{\vec}[1]{ \boldsymbol{#1}}

\usepackage{diagbox}

\title{Efficient solution of symmetric eigenvalue problems from families of
  coupled systems \thanks{\funding{The first
      author was partially supported by the Stenb\"{a}ck foundation, the second
      author by the Magnus Ehrnrooth foundation, and the third author by the
      V\"{a}is\"{a}l\"{a} foundation along with the Academy of Finland (312340).
    }}}

\author{Antti Hannukainen\thanks{Department of Mathematics and Systems Analysis, Aalto University
  (\email{antti.hannukainen@aalto.fi}, \email{jarmo.malinen@aalto.fi}, \email{antti.ojalammi@aalto.fi})}
\and Jarmo Malinen\footnotemark[2] 
\and Antti Ojalammi\footnotemark[2]}

\usepackage{amsopn}

\ifpdf \hypersetup{ pdftitle={Efficient solution of symmetric
    eigenvalue problems from families of coupled systems},
  pdfauthor={A. Hannukainen, J. Malinen, and A. Ojalammi} } \fi

\begin{document}

\maketitle

\begin{abstract}
  Efficient solution of the lowest eigenmodes is studied for a family of related
  eigenvalue problems with common $2\times 2$ block structure. It is assumed
  that the upper diagonal block varies between different versions while the
  lower diagonal block and the range of the coupling blocks remains unchanged.
  Such block structure naturally arises when studying the effect of a subsystem
  to the eigenmodes of the full system. The proposed method is based on
  interpolation of the resolvent function after some of its singularities have been
  removed by a spectral projection. Singular value decomposition can be used to
  further reduce the dimension of the computational problem. Error analysis of
  the method indicates exponential convergence with respect to the number of
  interpolation points. Theoretical results are illustrated by two numerical
  examples related to finite element discretisation of the Laplace operator.
\end{abstract}




\section{Introduction}
There is often a need to study the effect of a subsystem to the vibration modes
of the whole system. For example, consider the modal computations of a vocal
tract constrained into a Magnetic Resonance Imaging (MRI)
scanner~\cite{Hannukainen:2007, Kuortti:Sigproc2016}. In this case, the system
consists of the vocal tract air volume (i.e., \emph{the interior system}) that
changes during speech, and the air volume of the MRI head coil (i.e., \emph{the
  exterior system}) that stays unchanged, see Figure \ref{fig:vt_mri}. For high
resolution description of speech production, it is desirable to compute the
resonances for a very large number of vocal tract shapes. In order to speed up these
computations, there is a strong incentive to precompute the effect of the
unchanging exterior system and use it efficiently.

Modal analysis of systems consisting of interior and exterior parts
leads to an algebraic eigenvalue problem
\begin{equation} \label{eq:AlgEigValP}
  A \vec{x} = \lambda M \vec{x} \quad \mbox{with the normalisation} \quad \vec{x}^T M\vec{x} = 1
\end{equation}
that can accordingly be decomposed as
\begin{equation}
\label{eq:problem1}
\begin{bmatrix} A_{11} & A_{12} \\ A_{21} & A_{22} \end{bmatrix} \begin{bmatrix} \vec{x}_1 \\ \vec{x}_2 \end{bmatrix} 
= \lambda \begin{bmatrix} M_{11} & M_{12} \\ M_{21} & M_{22} \end{bmatrix} \begin{bmatrix} \vec{x}_1 \\ \vec{x}_2 \end{bmatrix}.
\end{equation}
Here the matrix blocks $A_{11}$ and $A_{22}$ refer to interior and exterior
systems, respectively, and the matrix blocks $A_{12}$ and $A_{21}=A_{12}^T$ are related
to the coupling between the two systems. The same descriptions hold for the
matrix $M$. In the following, we assume that the matrices $A$ and $M$ are large,
sparse, symmetric, and positive definite, implying the same properties for
$A_{ii}$ and $M_{ii}$ for $i=1,2$. This assumption is satisfied, e.g., when
problem~\eqref{eq:AlgEigValP} is related to the finite element discretisation of
an elliptic PDE.

In this article, a novel method is proposed for efficiently solving a large
number of different versions of problem~\eqref{eq:problem1} for the smallest
eigenvalues $\lambda \in (0,\Lambda)$, $\Lambda > 0$, together with the
corresponding eigenvectors. In applications, the number of eigenvalues in
$(0,\Lambda)$ is typically much smaller than the dimension of the full problem.
It is assumed that the matrices $A_{22}, M_{22}$ and subspaces
$\mathop{range}(A_{21})$, $\mathop{range}(M_{21})$ remain unchanged while the
matrix $A_{11}$ varies between different versions of the
problem~\eqref{eq:problem1}. In the proposed method, the matrix blocks related
to the exterior system are replaced by ones with considerably smaller dimension.
As shown in Section~\ref{sec:num_ex}, the time required to solve eigenvalues of
interest for the acoustic system shown in Figure~\ref{fig:vt_mri} is reduced
from $25$ to $5$ seconds, not accounting for precomputation time. In a family of
$1\,000$ different vocal tract samples this constitutes a saving of over five
hours. After $40$ eigensolves, the proposed method is faster even when the
precomputational time is taken into account.

There exists a considerable amount of literature on the solution of large,
sparse, symmetric and positive definite eigenvalue problems,
see~\cite{Parlett:1998}. The state-of-the-art solution method for this class of
problems is the Lanczos iteration, which is a Rayleigh--Ritz method based on
solving the eigenvalue problem in a Krylov subspace. When the interest lies in
the smallest eigenvalues, the convergence of the iteration is sped up by using
the shift-and-invert strategy, i.e., considering the eigenvalue problem related
to the matrix $(A+\sigma M)^{-1}$ for some $\sigma \in \mathbb{R}$ instead.

As such, the Lanczos iteration is not well suited for including precomputations
involving the exterior system. Computing the lowest eigenmodes using
shift-and-invert strategy requires the action of $(A+\sigma M)^{-1}$ in each
iteration step. As several linear systems need to be solved, the matrix
$(A+\sigma M)$ is typically factorised, e.g, using the $LDL^T$ factorisation.
Unfortunately, all factorisations have to be recomputed for different versions
of Eq.~\eqref{eq:problem1}. In doing so, the block structure of the problem
should be taken into account; see Section~\ref{sec:cost} for an example in
recycling information in computing block Cholesky factorisations. However, such
a strategy does not easily allow for dimension reduction in (the exterior part
of) the eigenvalue problem.

We propose a \emph{condensed pole interpolation} (CPI) method that is based on
the Rayleigh--Ritz procedure. In CPI, a subspace related to the exterior part of
the problem is precomputed by a combination of a spectral projection, Chebyshev
interpolation of the resolvent after removal of poles, and dimension reduction
using singular value decomposition (SVD). This subspace is constructed only
once, and it can be reused for different versions of $A_{11}$. For each version
of Eq.~\eqref{eq:problem1}, one solves a much smaller symmetric, positive
definite eigenvalue problem using, e.g., the Lanczos iteration with the
shift-and-invert strategy. Dimension reduction using SVD in the context of
eigenvalue problems has been studied, e.g., in~\cite{Fumagalli:2016,
  Buchan:2013,Grabner:2016}.

Our approach has some similarities with the component mode synthesis (CMS)
introduced in the 1960's as a substructuring method for engineering
simulations~\cite{hurty1960vibrations,bampton:coupling_1968}. An error estimate
for the original CMS is given in \cite{voss:amls:2006,Bourquin1992}, error
indicator has been studied in \cite{Kim:Rel_error:2014,Kim:Error_AMLS:2016}, and
more efficient variants have been introduced, e.g., in
\cite{Park:Partitioned:2004,Kim:Enhanced_CB:2015,Kim:Structural:2017,saad:amls}.
The CMS method has been further developed into automated multi-level
substructuring (AMLS) method having a much smaller precomputational cost without
loss of accuracy, see \cite{Benninghof:AMLS:2004, Kim:Enhanced_AMLS:2015}. The
rational filtering domain decomposition eigenvalue solver (RF-DDES) has recently
been proposed in \cite{Saad:2018} for computing eigenvalues in a spectral
interval of interest by using a Neumann series approximation of the resolvent
function. The underlying philosophy of CPI is similar to RF-DDES.

The outline of the work is as follows. The required background is reviewed in
Section~\ref{sec:prel}, and CPI is introduced together with its error analysis
in Section~\ref{sec:method}. Eigenvalue error estimates are given in
Section~\ref{sec:finalbounds}, and the optimal selection of the two parameter
values, required by CPI, is discussed in Section~\ref{sec:cost}. Further
dimension reduction is the matter of Section~\ref{sec:dim_red}. Finally, the
theoretical treatment is illustrated in numerical examples in
Section~\ref{sec:num_ex}.

\begin{figure}[h]
  \centering
  \includegraphics[height=.36\textheight]{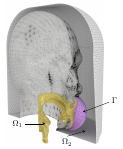}
  \includegraphics[height=.36\textheight]{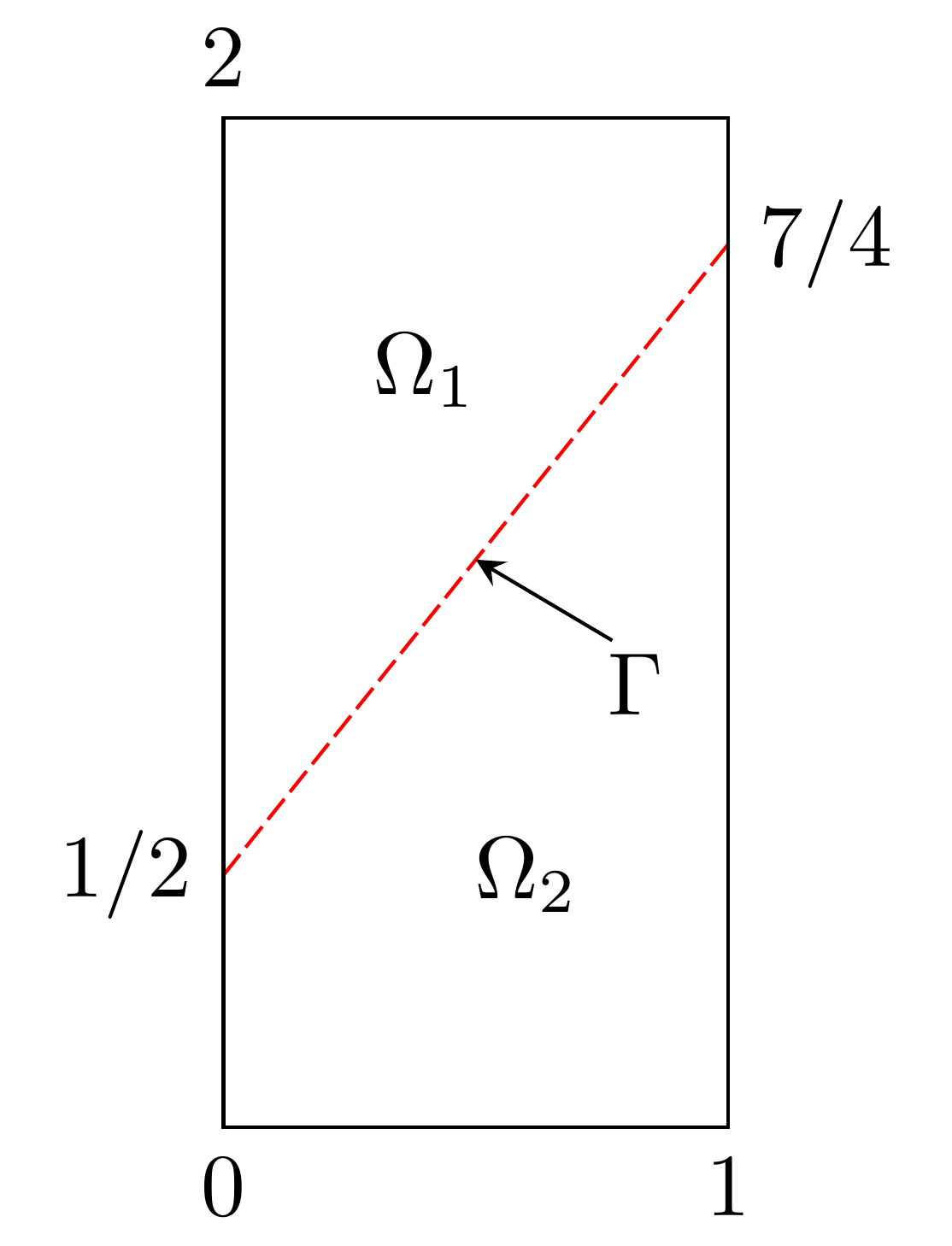}
  \caption{The domains considered in the article. Left: an acoustic system with
    a human vocal tract inside an MRI head coil. The interior domain $\Omega_1$ is
    connected to the exterior domain $\Omega_2$ via the interface $\Gamma$. The
    exterior domain is symmetric around the cross section. Right:
    two-dimensional rectangular domain with a non-symmetric diagonal interface
    marked by the dashed line.}
  \label{fig:vt_mri}
\end{figure}

\section{Background}
\label{sec:prel}
Let $A, M \in \mathbb{R}^{n \times n}$ be symmetric, positive definite
matrices.  Let $(\lambda,\vec{x}) \in \mathbb{R}\times \mathbb{R}^{n}
\setminus \{ 0 \}$, $\lambda > 0$, be a solution of the \emph{full}
symmetric eigenvalue problem $A\vec{x} = \lambda M\vec{x}$ such that
$\vec{x}^T M \vec{x} = 1$. For such $\lambda$'s, we write
\begin{equation} \label{eq:MatrixPencilSpectrum}
  \lambda \in \sigma(A,M) := \left \{ z \in \mathbb{C} \, : \, A - z M \text{ is not invertible} \right \}.
\end{equation}
From now on, denote $\vec{x} = [ \vec{x}_1 \; \vec{x}_2 ]^T$ where
$\vec{x}_1 \in \mathbb R^{n_1},\; \vec{x}_2 \in \mathbb R^{n_2}$ with
$n = n_1 + n_2$.  We call this the \emph{standard splitting} of
$\mathbb{R}^{n}$ where $\mathbb R^{n_1}$ and $\mathbb R^{n_2}$ are
called \emph{interior} and \emph{exterior spaces}, respectively. Using
the standard splitting, the full eigenvalue problem has the structure
\begin{equation}
  \label{eq:dec_sys}
  \begin{bmatrix} A_{11} & A_{12} \\ A_{21} & A_{22} 
  \end{bmatrix}
  \begin{bmatrix} \vec{x}_1 \\ \vec{x}_2 \end{bmatrix} = 
  \lambda
  \begin{bmatrix} M_{11} & M_{12} \\ M_{21} & M_{22} \end{bmatrix}
  \begin{bmatrix} \vec{x}_1 \\ \vec{x}_2 \end{bmatrix}
\end{equation} 
where the blocks $A_{11},M_{11} \in \mathbb{R}^{n_1 \times n_1}$,
$A_{22},M_{22} \in \mathbb{R}^{n_2 \times n_2}$, $A_{12},M_{12} =
\mathbb{R}^{n_1 \times n_2}$, and $A_{21},M_{21} = \mathbb{R}^{n_2
  \times n_1}$.

The topic of this work is the solution of different versions of the eigenvalue
problem \eqref{eq:dec_sys} where the matrices $A_{11}$, $M_{11}$ are free to vary
but the matrices $A_{22}$, $M_{22}$ and subspaces
$\mathop{range}(A_{21})$, $\mathop{range}(M_{21})$ stay the same. In this case,
one can afford even expensive precomputations for the unchanging components as a
part of the eigenvalue solution method.
 
As an example, consider the acoustic system shown in
Figure~\ref{fig:vt_mri}. In this case, the eigenvalue problem
\eqref{eq:dec_sys} arises from finite element discretisation of the
variational eigenvalue problem: Find $(\lambda' , u) \in \mathbb{R}
\times \mathcal{V}$ such that
\begin{equation} \label{eq:cont_eigen}
(\nabla u,\nabla v) = \lambda' (u,v)  \quad \text{ for all } \quad v \in \mathcal{V}
\end{equation}
where $(\cdot,\cdot)$ is the inner product of $L^2(\Omega)$, and the
subspace $\mathcal{V} \subset H^1(\Omega)$ enforces the homogeneous Dirichlet
boundary condition at least on a part of the boundary $\partial
\Omega$. Then the resulting $A,M \in \mathbb R^{n \times n}$ are
symmetric, positive definite stiffness and mass matrices,
respectively.

The standard splitting in \eqref{eq:cont_eigen} arises from
decomposition of the domain $\Omega$ into non-overlapping subdomains
$\Omega_1$ and $\Omega_2$, corresponding to varying and unchanging
parts of the system, respectively. The interior interface between the
two subdomains is denoted by $\Gamma = \left( \partial \Omega_1 \cup
\partial \Omega_2 \right) \setminus \partial \Omega$. The vectors
$\vec x_1 \in \mathbb{R}^{n_1}$ and $\vec{x}_2 \in \mathbb{R}^{n_2}$
correspond to the degrees of freedom of the finite element space on
$\Omega$ corresponding to $\Omega_1$ and $\Omega_2$, respectively. In
addition to $n_1$ and $n_2$, we define a third characterising integer
\begin{equation}
\label{eq:defNG}
n_\Gamma := \mathop{dim}( \mathop{range}( \begin{bmatrix} M_{21} & A_{21} \end{bmatrix})) 
\end{equation}
which gives the number of degrees of freedom over which the interior and the
exterior systems interact on the interface~$\Gamma$. The FEM discretisation of
the full domain $\Omega$ can be carried out in many ways, and the interface
$\Gamma$ need not be consistent with the FEM mesh. However, the three numbers
$n_1, n_2,n_\Gamma$ can always be extracted from the standard splitting.

\subsection{Subspace Methods}

Most solution methods for eigenvalue problems are of Rayleigh-Ritz type in which
the eigenvalue problem is projected to a given subspace of
$\mathbb{R}^n$~\cite{Parlett:1998}. For this purpose, let $Q \in \mathbb R^{n
  \times k}$, $k \leq n$, be a \emph{method matrix} with linearly independent
column vectors that is used for defining the \emph{method subspace} $V :=
\mathop{range}(Q)$. Poor conditioning in numerical realisations is avoided by
choosing the column vectors of $Q$ orthonormal in an appropriate inner product;
see Section~\ref{sec:dim_red}.

In the Rayleigh--Ritz procedure, the
eigenvalue problem in $V$ is posed as follows: find $(\tilde{\lambda},
\tilde{\vec{x}}) \in \mathbb{R} \times \mathbb{R}^k$ such that
\begin{equation}
  \label{eq:subspace_eigenvalue}
Q^T A Q \tilde{\vec x} = \tilde{\lambda} Q^T M Q \tilde{\vec{x}}. 
\end{equation}
 The set of approximate eigenvalues $\tilde \lambda$ is denoted by
 $\sigma_Q( A,M ) := \sigma(Q^T AQ, Q^T M Q )$ as in
 Eq.~\eqref{eq:MatrixPencilSpectrum}. In fact, the set $\sigma_Q( A,M )$ depends 
 only on the method subspace $V$:
 
\begin{lemma}
Let $A,M \in \mathbb{R}^{n\times n}$ be symmetric and positive
definite. In addition, let $Q_1,Q_2 \in \mathbb{R}^{n\times k},k<n$ be
such that $\mathop{range}(Q_1) = \mathop{range}(Q_2)$. Then
\begin{equation}
\sigma_{Q_1}(A,M) = \sigma_{Q_2}(A,M).
\end{equation}
\end{lemma}
\noindent Hence, we can write $\sigma_Q(A,M) = \sigma_{V}(A,M)$ where $V =
\mathop{range}(Q)$. The aim is to find a low dimensional subspace $V$ such that
$\sigma_V( A,M )$ is a reasonable approximation for a relevant part of $\sigma(
A,M )$. Those eigenvalues can be computed using, e.g., the shift-and-invert
Lanczos iteration~\cite{ericsson:lanczos:1980}.

\subsection{Estimate for the relative eigenvalue error}
\label{sec:projections}

The relative error between corresponding eigenvalues in $\sigma(A,M)$ and
$\sigma_V(A,M)$ is estimated by studying approximation of eigenvectors in the
method subspace $V$:
\begin{proposition}
  \label{prop:eigen_proj_estimate} Let $A,M \in \mathbb{R}^{n \times n}$ be
  symmetric and positive definite matrices. Let $(\lambda,\vec{x}) \in
  \sigma(A,M) \times \mathbb{R}^{n} \setminus \{ 0 \}$ be an eigenpair of
  Eq.~\eqref{eq:AlgEigValP} corresponding to a simple eigenvalue $\lambda$ such
  that $\vec{x}^T M \vec{x} = 1$. In addition, assume that the spectral gap condition in  \cite[Th.2.7]{KnOs:07} is satisfied. Then for any $V \subset \mathbb{R}^{n}$ there
  exists $\tilde{\lambda} \in \sigma_V(A,M)$ and $C(\lambda):= C(\lambda;A,M,V)$ such that
  \begin{equation}
  \label{eq:bestapprox} \frac{| \lambda - \tilde{\lambda} |}{\lambda} \leq
C(\lambda) \min_{\vec{v} \in V} \| \vec x  - \vec v \|_A^2  \quad \text{where } \quad \| \vec{x} \|_A : = \| A^{1/2} \vec{x} \|_2.
 \end{equation} 
\end{proposition}
\noindent This proposition is a special case of a Hilbert space result in
\cite{KnOs:07}; see also \cite{BaOs:1989,Chatelin:1975,Bo:2010}. We normalise the eigenvectors $\vec{x}$ as $\vec{x}^T M \vec{x} = 1$ instead of using $\vec{x}^T A \vec{x} = 1$ from \cite{KnOs:07}. Because
\begin{equation*}
\min_{\vec{v} \in V} \|\vec{x}-\vec{v}\|^2_A =  \lambda^{-1} \min_{\vec{v} \in V} \left \| \frac{\vec{x}}{\| M^{1/2} \vec{x} \|_2} - \vec{v} \right\|^2_A
\end{equation*}
the different normalisations can be absorbed in $C(\lambda)$. Otherwise the multiplier $\lambda^{-1}$ will appear in Eq.~\eqref{BetaEstimates}. The eigenvector error can be similarly related to the angle between the exact eigenvector and the method subspace, see \cite{KnOs:07}. Except for the term $\lambda^{-1}$ due to normalisation, the dependency of $C(\lambda;A,M,V)$ on its parameters is explained in \cite[Th.2.7]{KnOs:07}.

 Henceforth, the method subspace is required to satisfy
\begin{equation}
\label{eq:subspace_V}
V = \left \{ \; \begin{bmatrix} \vec{v}_1 \\ \vec v_2 \end{bmatrix} \; \bigg| \; \vec{v}_1 \in \mathbb{R}^{n_1} \text{ and } \vec{v}_2 \in V_2 \; \right\} \quad \textrm{ where } \quad V_2 \subset \mathbb{R}^{n_2}.
\end{equation}
Let $\vec{x}$ in Proposition~\ref{prop:eigen_proj_estimate}~be decomposed as in
equation~\eqref{eq:dec_sys}, i.e., $\vec{x} = \begin{bmatrix} \vec{x}_1 &
  \vec{x}_2 \end{bmatrix}^T$. Choosing $\vec{v}$ in Eq.~\eqref{eq:bestapprox} so
that $\vec{v}_1 = \vec{x}_1$ and using
Proposition~\ref{prop:eigen_proj_estimate} leads to
\begin{equation}
\label{eq:rel_error}
\frac{| \lambda - \tilde{\lambda} |}{\lambda} \leq C(\lambda) \min_{\vec v_2 \in V_{2} }  \| \vec x_2 - \vec{v}_2  \|^2_{A_{22}}.
\end{equation}
We conclude that a subspace~$V_2$~should accurately represent the $\vec
x_2$-component of eigenvectors $\vec{x}$ for $\lambda \in (0, \Lambda)$. In the
proposed method, this approximation is guaranteed by constructing $V_2$ using a combination
of spectral projection and Chebyshev interpolation of the resolvent.
 
\begin{remark}
  In the case of multiple eigenvalues, the relative error in eigenvalue
  $\lambda$ is related to the maximum over the corresponding eigenspace
  $E_\lambda$:
\begin{equation*} \frac{| \lambda - \tilde{\lambda} |}{\lambda} \leq
C(\lambda) \max_{\substack{\vec{x} \in E_\lambda \\ \vec{x}^T M  \vec{x} = 1}}  \min_{\vec v \in V} \| \vec x - \vec v \|_A^2.
\end{equation*} 
All upcoming results generalise to multiple eigenvalues by replacing $\vec{x}$
with $\hat{\vec{x}}$ such that
\begin{equation*}
 \max_{\substack{\vec{x} \in E_\lambda \\ \vec{x}^T M  \vec{x} = 1}}  \min_{\vec v \in V} \| \vec x - \vec v \|_A =   \min_{\vec v \in V} \| \hat{\vec x} - \vec v \|_A.
\end{equation*}
\end{remark}
For notational convenience, we assume in the following that all eigenvalues are simple.

\subsection{Method matrix in Component Mode Synthesis}
\label{sec:cms}

In CMS, the domain $\Omega$ is decomposed into several subdomains, and the matrix $A$ is partitioned according to the degrees of freedom corresponding to
the subdomains and the subdomain interfaces. After
partitioning, the matrix $A$ is block diagonalised using an appropriate
elimination matrix. In the last step, the block corresponding to the subdomain
degrees of freedom is truncated by using a select number of eigenvectors of each
local eigenvalue problem.

It is straightforward to adapt CMS to deal with the standard splitting in
Eq.~\eqref{eq:dec_sys} and to perform the dimension reduction only on the
exterior domain. The method matrix $Q$ is constructed as a product of an
elimination matrix
$G := \left[ \begin{smallmatrix} I & 0 \\
    - A_{22}^{-1} A_{21} & I \end{smallmatrix} \right]$ that block diagonalises
the matrix $A$ and a matrix containing eigenvectors related to the $K \ll n_2$
smallest eigenvalues of the subproblem $A_{22} \vec{x}_2 = \lambda M_{22}
\vec{x}_2$. The resulting method matrix is
\begin{equation}
\label{eq:CMSMM}
  Q : = 
  \begin{bmatrix}
     I & 0 \\ - A_{22}^{-1} A_{21} & I
  \end{bmatrix}  
  \begin{bmatrix}
    I & 0 \\ 0  & [\vec{v}_1, \ldots, \vec{v}_K]
  \end{bmatrix}.
\end{equation}
When using $Q$ in Eq.~\eqref{eq:CMSMM}, numerical experiments in Section \ref{sec:num_ex} indicate similar performance as reported in \cite{voss:amls:2006,Bourquin1992} and larger computational effort compared to CPI.

\section{Condensed pole interpolation method}
\label{sec:method}

Assume that $A$ and $M$ are now represented through the standard splitting as in
Eq.~\eqref{eq:dec_sys}. The topic of this section is the construction of the
subspace $V_2$ in Eq.~\eqref{eq:subspace_V}. Consider
the eigenvalue $\lambda > 0$ as fixed, and define an additional bound $\tilde
\Lambda$ satisfying
\begin{equation}
\label{eq:defgamma}
  0 < \lambda < \Lambda < \tilde \Lambda = \gamma \Lambda \quad \text{ for } \quad  \gamma > 1
\end{equation}
where $(0,\Lambda)$ is the spectral interval of interest.

\subsection{Eigenvector basis for the exterior subspace}
Let $(\mu_k,\vec v_k) \in \mathbb{R}\times \mathbb{R}^{n_2} \setminus
\{ 0 \}$ be solutions of the symmetric \emph{exterior} eigenvalue problem,
such that
\begin{equation} \label{ExteriorEVP}
  A_{22} \vec v_k = \mu_k M_{22} \vec v_k \quad \textrm{and} \quad \vec{v}^T_j M_{22} \vec{v}_k  
  = \delta_{j,k} \text{ such that } 1 \leq j,k \leq n_2.
\end{equation}
(Note that $M_{22}$ is positive definite since $M$ is.)  For $\tilde{\Lambda}
> 0$, let $P_{\tilde \Lambda} \in \mathbb{R}^{n_2 \times n_2}$ be the
$M_{22}$-orthogonal projection matrix
\begin{equation}
  \label{eq:Pdef}
  P_{\tilde \Lambda} := \sum_{\{k \; : \; \mu_k \in (0,\tilde{\Lambda})\}} \vec v_k \vec v_k^T M_{22} \quad \text{ satisfying } \quad M_{22} P_{\tilde \Lambda}  = P_{\tilde \Lambda}^T M_{22}.
\end{equation}
We further restrict $V_2$ in Eq.~\eqref{eq:subspace_V}~to subspaces of the type
\begin{equation} \label{FurtherRestriction}
V_2 = \mathop{range}(P_{\tilde \Lambda}) \oplus W_2
\end{equation}
where the \emph{complementing subspace} $W_2 \subset \mathbb{R}^{n_2}$, $W_2 \perp\mathop{range}(P_{\tilde \Lambda})$ in the $A_{22}$-inner product will be
chosen so that the eigenvalue error given by Eq.~\eqref{eq:rel_error} can be
conveniently bounded from above.

\subsection{Error estimate based on projection and interpolation}
For $\lambda \in \sigma (A,M) \setminus \sigma (A_{22}, M_{22})$ and the
corresponding eigenvector $\vec{x} = [\vec{x}_1\; \vec{x}_2]^T$,
Eq.~\eqref{eq:dec_sys} gives
\begin{equation*}
\vec{x}_2 =  (A_{22} - \lambda M_{22})^{-1} Z(\lambda) \vec{x}_1 \quad \textrm{ where }  \quad Z(\lambda) := \lambda M_{21} - A_{21}.
\end{equation*}
Clearly, 
\begin{equation}
\label{eq:v2red1}
\begin{aligned}
  \min_{\vec{v}_2 \in V_2} \| \vec{x}_2 - \vec{v}_2 \|_{A_{22}}
                                                                &= \min_{\vec{v}_2 \in V_2} \| P_{\tilde \Lambda} \vec{x}_2 + (I-P_{\tilde \Lambda}) \vec{x}_2 - \vec{v}_2 \|_{A_{22}} \\
  &= \min_{\vec{v}_2 \in V_2} \| (I-P_{\tilde \Lambda}) (A_{22} - \lambda M_{22} )^{-1} Z(\lambda)\vec{x}_1 - \vec{v}_2 \|_{A_{22}} ,
\end{aligned}
\end{equation}
since $P_{\tilde \Lambda}(A_{22} - \lambda M_{22} )^{-1} Z(\lambda)\vec{x}_1 \in \mathop{range}(P_{\tilde \Lambda}) \subset V_2$
by Eq.~\eqref{FurtherRestriction}. 
Because also $W_2 \subset V_2$, it follows that
\begin{equation} \label{eq:apuW2}
\min_{\vec{v}_2 \in V_2} \| \vec{x}_2 - \vec{v}_2 \|_{A_{22}} \leq \min_{\vec{w}_2 \in W_2} \| (I-P_{\tilde \Lambda}) (A_{22} - \lambda M_{22} )^{-1} Z(\lambda)\vec{x}_1 - \vec{w}_2 \|_{A_{22}} .
\end{equation}

So as to introduce the CPI method, we proceed to construct the complementing
subspace $W_2$ for Eq.~\eqref{eq:apuW2} depending on $\tilde \Lambda$, some
distinct interpolation points $\{ \xi_i \}_{i=1}^N \subset (0,\Lambda)$, and
subspaces $\mathop{range}(M_{21})$ and $\mathop{range}(A_{21})$ related to the
standard splitting of the original matrices $A$ and $M$; i.e.,
\begin{equation*}
W_2 = W_2 \left(\tilde \Lambda, \{ \xi_i \}_{i=1}^N, \mathop{range}(\begin{bmatrix} M_{21} & A_{21} \end{bmatrix} )  \right).
\end{equation*}
Let
\begin{equation*}
  f_{\tilde \Lambda}(\xi) := (I-P_{\tilde \Lambda}) (A_{22} - \xi M_{22} )^{-1} \in \mathbb{R}^{n_2 \times n_2} \quad \textrm{for } \xi \in (0,\Lambda) \setminus \sigma(A_{22},M_{22}).
\end{equation*}
Then $A_{22} \vec v_k - \xi M_{22} \vec v_k = ( \mu_k - \xi) M_{22} \vec v_k$ where $(\mu_k, \vec v_k)$, $k = 1, \ldots , n_2$, are given by Eq.~\eqref{ExteriorEVP}. This
implies 
\begin{equation}
\label{eq:smap}
(A_{22} - \xi M_{22})^{-1} M_{22} \vec v_k= (\mu_k - \xi)^{-1} \vec v_k.
\end{equation}
Hence,
\begin{equation} \label{InterpolantExpansion}
\begin{aligned}
f_{\tilde \Lambda}(\xi) M_{22} \vec v_k & =  (I-P_{\tilde \Lambda}) (A_{22} - \xi M_{22})^{-1} M_{22} \vec v_k  \\  &=  (\mu_k - \xi)^{-1}  (I-P_{\tilde \Lambda}) \vec v_k  \\
 & =  (\mu_k - \xi)^{-1}  \left ( \vec v_k - \sum_{\{j \; : \; \mu_j \in (0,\tilde \Lambda)\}} \vec v_j \vec v_j^T M_{22} \vec v_k \right) \\
 & = \begin{cases} (\mu_k - \xi)^{-1} \vec v_k & \textrm{ for } k \text{ satisfying } \mu_k > \tilde \Lambda, \\ 0 & \textrm{ otherwise}. \end{cases}
\end{aligned}
\end{equation}
Because $\mathbb{R}^{n_2} = \mathop{span}_{k=1,\ldots,n_2} \{ {M_{22} \vec v_k}
\}$, we conclude that $f_{\tilde \Lambda}$ is, in fact, an analytic function on
the whole interval $(0,\tilde \Lambda)$ that contains the original domain $(0,\Lambda)
\setminus \sigma(A_{22},M_{22})$. Hence, the assumption $\lambda \notin \sigma(A_{22},M_{22})$ can be removed. As $f_{\tilde \Lambda}$ is analytic, it can be
approximated in various ways such as series expansions or interpolation.

In the CPI method, the complementing subspace $W_2$ is chosen so that the right hand side of Eq.~\eqref{eq:apuW2} can be bounded using interpolation error estimates. Let $\{
\xi_i \}_{i=1}^N\subset (0,\Lambda)$ be a set of distinct interpolation points
and
\begin{equation}
\label{eq:defW2}
W_2 := \mathop{span}_{i=1,\ldots N } \left\{ f_{\tilde \Lambda}(\xi_i) \begin{bmatrix} M_{21} & A_{21} \end{bmatrix} \right\}
\end{equation}
In addition, let
\begin{equation}
\label{eq:ngKl}
K(l) := \# \{ \; \mu_k \; | \; \mu_k < l \; \}.
\end{equation}
The dimensions of the spaces $W_2$ and $V_2$ depend on the number of
interpolation points $N$, the number $K(\tilde \Lambda)$ of exterior eigenvalues
$\mu_k$ smaller than $\tilde \Lambda$, and $n_\Gamma$ defined in
Eq.~\eqref{eq:defNG}. Then $\mathop{dim}(W_2) \leq N n_\Gamma$ and $\mathop{dim}(V_2)
\leq N n_\Gamma + K(\tilde \Lambda)$.

\subsection{Outline of the method}
\label{sec:mmatrix}
The CPI method introduced above is based on solving the original
eigenvalue problem by restricting the exterior system to the space $V_2 = \mathop{range}(P_{\tilde \Lambda}) \oplus W_2$ where
\begin{equation*}
  W_2 = \mathop{span}_{i = 1,\ldots,N} \left\{ f_{\tilde \Lambda}(\xi_i)
    \begin{bmatrix} M_{21} & A_{21} \end{bmatrix} \right\}, \quad
f_{\tilde \Lambda}(\xi) = (I-P_{\tilde \Lambda}) (A_{22} - \xi M_{22}
)^{-1},
\end{equation*}
and $\xi_i \in (0,\Lambda)$, $i = 1, 2, \ldots , N$, are the
interpolation points chosen as in Eq.~\eqref{eq:defChebyIP}. Practical realisation of CPI requires a method matrix $Q$ which, by Eq.~\eqref{eq:subspace_V}, has the structure
\begin{equation} \label{eq:ourmethodmatrix}
  Q =
  \begin{bmatrix}
    I & 0 \\ 0 & Q_{22}
  \end{bmatrix} 
\end{equation}
where $I \in \mathbb{R}^{n_1\times n_1}$ is the identity matrix and
$V_2 = \mathop{range}(Q_{22})$. The column vectors of $Q_{22}$ form a basis of the space $V_2$ and are constructed with the aid of the \emph{sample vectors} $\vec{q}_{ij} \in \mathbb{R}^{n_2}$. Let $\{ \vec p_1, \ldots \vec p_r \} \subset \mathbb{R}^{n_2}$ be a
set of (possibly linearly dependent) vectors such that
\begin{equation} \label{pvectors}
\mathop{span}\{\vec p_1, \ldots \vec p_r \} = \mathop{range}(\begin{bmatrix} M_{21} &
  A_{21} \end{bmatrix}).
\end{equation}
The sample vectors are computed by solving the linear systems
\begin{equation}
\label{SampleVectors}
(A_{22} - \xi_i M_{22}) \vec{q}_{i j} = \vec p_j,
\end{equation}
and the complementing subspace $W_2$ is given by
\begin{equation*}
  W_2 = \mathop{span} \left \{ (I-P_{\tilde \Lambda}) \vec{q}_{ij} 
  \; | \; i = 1,\ldots, N \text{ and } j = 1,\ldots, r \right \}.
\end{equation*}  

\noindent Practical realisation of the CPI method consists of the following steps:
\begin{enumerate}
  
\item Compute the smallest eigenpairs $(\mu_k,
  \vec{v}_k)$ of the exterior system $A_{22} \vec v_k = \mu_k M_{22}
  \vec v_k$ satisfying $\mu_k \leq \tilde{\Lambda}$.

\item Compute sample vectors $\vec q_{11},\ldots,\vec q_{Nr}$ as solutions of $(A_{22} - \xi_i M_{22}) \vec{q}_{i j} = \vec p_j$ for $i=1,\ldots,N$ and $j=1,\ldots,r$. 
  
\item Collect the eigenvectors from Step~1 and the sample vectors from Step~2
  into matrix $B = [\vec v_1, \ldots, \vec v_K, (I-P_{\tilde{\Lambda}}) \vec
  q_{11}, \ldots, (I-P_{\tilde{\Lambda}}) \vec q_{Nr}]$. Use SVD to compute an
  orthonormal basis for $V_2$ from $B$. Use the basis vectors as
  columns of $\tilde{Q}$.
\item Solve the eigenvalue problem $\tilde Q^T A \tilde Q \tilde{\vec{x}} = \tilde{\lambda}
  \tilde Q^T M \tilde Q \tilde{\vec{x}}$ using, e.g., the Lanczos iteration.
\end{enumerate}
Step 3 will be modified to include an additional dimension reduction of $V_2$ in
Section~\ref{sec:dim_red}, which leads to a considerably smaller eigenvalue
problem while maintaining the desired accuracy.

\section{Bound for the relative eigenvalue error}
\label{sec:finalbounds}
We proceed to give an upper bound for the relative error.

\begin{lemma} 
\label{lemma:help1}
Let $V_2$ be as defined in Eqs.~\eqref{eq:subspace_V},~\eqref{eq:Pdef}~and~\eqref{eq:defW2}. Denote the Lagrange interpolating polynomials by
\begin{equation*}
\ell_i(\lambda) = \prod_{\substack{1 \leq j \leq N \\ j \neq i}} \frac{\lambda - \xi_j}{\xi_i - \xi_j} \quad \textrm{ for } i=1,\ldots,N
\end{equation*} 
where $\{\xi_i \}_{i=1}^N \subset (0,\Lambda)$ are the interpolation points used to define $W_2$. Then for any $\lambda \in \sigma(A,M) \cap (0,\Lambda)$ there exists $\tilde{\lambda} \in \sigma_V(A,M)$ such that
\begin{equation*}
\frac{ | \lambda - \tilde{\lambda} |}{\lambda} 
\leq C(\lambda) \sum_{\substack{k \\ \mu_k > \tilde \Lambda}} \mu_k c_k(\lambda)^2 \beta_k(\lambda)^2 \quad \textrm{for all} \quad \tilde \Lambda > \Lambda
\end{equation*}
where the coefficients $\beta_k(\lambda)$ and $c_k(\xi)$ are defined by
\begin{equation}
\label{eq:def_bk_ck}
Z(\lambda) \vec x_1 = \sum_{k=1}^{n_2} \beta_k(\lambda) M_{22} \vec v_k \quad \textrm{and} \quad 
c_k(\xi) := \left(\frac{1}{\mu_k - \xi} - \sum_{j=1}^N {\frac{\ell_j (\xi)}{\mu_k - \xi_j}} \right).
\end{equation}
\end{lemma}
\noindent Observe that the coefficients $c_k(\xi)$ are the error functions in Lagrange interpolation at points $\{ \xi_i \}_{i=1}^N$ of the rational function $(\mu_k - \lambda)^{-1}$,  and they are analytic functions for all $\xi \notin \sigma(A_{22}, M_{22})$. Note that if $\{ k
\,: \, \mu_k > \tilde \Lambda \} = \emptyset$, then
Eq.~\eqref{AlmostFinalRelativeErrorEstimate} gives $\lambda =
\tilde{\lambda}$.

\begin{proof}
To obtain an upper bound for relative eigenvalue error in
Eq.~\eqref{eq:rel_error}, we choose
\begin{equation}
\label{eq:w2def}
\vec w_2 = \sum_{i=1}^N \ell_i(\lambda) f_{\tilde \Lambda}(\xi_i) Z(\lambda) \vec x_1 \in W_2.
\end{equation}
in Eq.~\eqref{eq:apuW2}, giving
\begin{equation} \label{AlmostFinalRelativeErrorEstimate}
\frac{ | \lambda - \tilde{\lambda} |}{\lambda} 
\leq C(\lambda) \left  \| \left( f_{\tilde \Lambda}(\lambda) - \sum_{i=1}^N \ell_i(\lambda) f_{\tilde \Lambda}(\xi_i) \right) Z(\lambda) \vec x_1 \right \|^2_{A_{22}}.
\end{equation}
The term $Z(\lambda) \vec x_1 \in \mathbb{R}^{n_2}$ has the expansion
\begin{equation}
\label{eq:Zexp}
Z(\lambda) \vec x_1 = \sum_{k=1}^{n_2} \beta_k(\lambda) M_{22} \vec v_k 
\end{equation}
where $\vec v_k$ are given by Eq.~\eqref{ExteriorEVP}.  Using
Eq.~\eqref{InterpolantExpansion} gives
\begin{equation*}
     \left( f_{\tilde \Lambda}(\lambda) - \sum_{i=1}^N \ell_i(\lambda) f_{\tilde \Lambda}(\xi_i) \right) Z(\lambda) \vec x_1  =  \sum_{k, \mu_k > \tilde \Lambda} c_k(\lambda) \beta_k(\lambda) \vec v_k 
\end{equation*}
where $\tilde \Lambda = \gamma \Lambda$ for $\gamma > 1$. By using the exterior eigenvector basis Eq.~\eqref{ExteriorEVP}, 
we get
\begin{equation*}
\Bigg\| \sum_{\substack{k \\ \mu_k > \tilde \Lambda}} c_k(\lambda) \beta_k(\lambda) \vec v_k  \Bigg \|^2_{A_{22}} = \sum_{\substack{k \\ \mu_k > \tilde \Lambda}} \mu_k c_k(\lambda)^2 \beta_k(\lambda)^2
\end{equation*}
which completes the proof.
\end{proof}

To estimate the relative error from
Eq.~\eqref{AlmostFinalRelativeErrorEstimate}, it only remains to bound
$c_k(\lambda)$ and $\beta_k(\lambda)$ from above.  In order to obtain a good upper bound for the functions
$c_k(\lambda)$, it is beneficial to choose the interpolation points as
zeroes of the Chebyshev polynomials on the interval
$(0,\Lambda)$
\begin{equation}
\label{eq:defChebyIP}
\xi_i = \frac{\Lambda}{2} \left[1+\cos{\left(\frac{2i-1}{2N}\pi\right)}\right], \quad i=1,\ldots,N.
\end{equation}
Then the functions $c_k(\lambda)$ defined in Eq.~ \eqref{eq:def_bk_ck}~can be uniformly bounded
on $(0, \Lambda)$ by the standard Lagrange error estimates
\begin{equation} \label{LagrangeErrorEstimates}
\begin{aligned}
   \sup_{\lambda \in (0,\Lambda)}{\left |c_k(\lambda) \right |} 
   	&  \leq \frac{\Lambda^N}{2^{2N-1} N!} \,
   \sup_{\xi \in (0,\Lambda)} \left | \frac{d^N}{d \xi^N } (\mu_k - \xi)^{-1} \right | \\
  &  =   \frac{\Lambda^N}{2^{2N-1}}   \sup_{\xi \in (0,\Lambda)}{(\mu_k - \xi)^{-N-1}}
   = \frac{\Lambda^{N}}{2^{2N-1}(\mu_k - \Lambda)^{N+1}}. 
   \end{aligned}
\end{equation}
To bound the coefficients $\beta_k(\lambda)$, we need a technical lemma:
\begin{lemma} \label{RestrictionEstimateLemma}
Let $M =   \left [\begin{smallmatrix} 
    M_{11} & M_{12} \\ M_{21} & M_{22}
  \end{smallmatrix} \right ]$ be a symmetric, positive definite matrix. Then
\begin{equation*}
  C_M \vec{x}^T M \vec{x} \geq  \| \vec{x}_2 \|_{M_{22}}^2 \quad \text{ for all }  \quad 
\vec{x} =
 \begin{bmatrix} \vec{x}_1 \\ \vec{x}_2 \end{bmatrix}
\end{equation*}
where 
\begin{equation} \label{RestrictionEstimateLemmaConstant}
  C_M^{-1} := \min{\left \{ \eta > 0 \, : \, \eta \in \sigma(I - M_{22}^{-1/2} M_{12}^T M_{11}^{-1} M_{12} M_{22}^{-1/2}) \right \} }.
\end{equation}
\end{lemma}
\begin{proof} Since $M$ is positive definite, so are $M_{11}$
  and $M_{22}$. Defining $M_k := \left
    [\begin{smallmatrix} M_{11} & M_{12} \\ M_{21} & (1 - k) M_{22}
    \end{smallmatrix} \right ]$ for $k \geq 0$ we observe that $\vec{x}^T M
  \vec{x} \geq k \| \vec{x}_2 \|_{M_{22}}^2$ for all $\vec{x}$ if and only if
  $M_k \geq 0$. Since $M_0 = M$, the set $\{ k > 0 \, : \, M_k \geq 0 \}$ is
  nonempty by the continuity of the eigenvalues of the matrix elements and the
  fact that the set of invertible matrices is open. Hence, we can define $\tilde
  C_M := \max{ \{ k > 0 \, : \, M_k \geq 0 \}}$. Similarly, we may reason that
  the matrix $M_{k}$ for $k = \tilde C_M$ is not invertible but it satisfies
  $M_k \geq 0$.

  For any $\eta \in \mathbb{R}$ the matrix $M_\eta$ is not invertible if and only if
  $M_\eta \vec{x} = 0$ for some $\vec{x} \neq \vec{0}$ if and only if
  \begin{equation*}
    \begin{aligned}
      \eta \in 1 - \sigma(M_{21} M_{11}^{-1} M_{12} M_{22}^{-1})
      & \subset 1 - \sigma(M_{21} M_{11}^{-1} M_{12} M_{22}^{-1}) \cup \{ 0 \} \\
      &  = \sigma(I - M_{22}^{-1/2} M_{12}^T M_{11}^{-1} M_{12} M_{22}^{-1/2}) \cup \{ 1 \}
      \subset (-\infty, 1]
    \end{aligned}
  \end{equation*}
  since $M_{22}^{-1/2} M_{12}^T M_{11}^{-1} M_{12} M_{22}^{-1/2} \geq
  0$. We used here the fact that $\sigma(AB) \cup \{0\} = \sigma(BA)
  \cup \{0\}$ for all square matrices $A$ and $B$.  Defining now $C_M$
  by Eq.~\eqref{RestrictionEstimateLemmaConstant}, we observe that
  $C_M^{-1} \leq \tilde C_M$, and the proof is thus complete.
\end{proof}
\begin{remark} 
  Note that if $M_{12} = 0$, then $C_M = 1$. We leave it to the reader to verify
  that the estimate in Lemma~\ref{RestrictionEstimateLemma} is, in fact, sharp.
  This can be seen by checking that an equality $C_M^{-1} = \tilde C_M$ holds in
  the proof.
\end{remark}

\noindent We have now completely specified the CPI method together with its
error estimate, and we are in the position to state our first main result:

\begin{theorem}
\label{th:FinalEstimate} 
Let $A$ and $M$ be as defined in Eq.~\eqref{eq:AlgEigValP}, $\{ \xi_i\}_{i=1}^N$
be the Chebyshev interpolation points of the interval $(0,\Lambda)$ given in
Eq.~\eqref{eq:defChebyIP}, $V_2 = \mathop{range}(P_{\tilde \Lambda}) \oplus W_2$
where $W_2$ and $P_{\tilde \Lambda}$ are as defined in
Eqs.~\eqref{eq:Pdef}~and~\eqref{eq:defW2}, respectively. Finally, let $V$ be as
defined in Eq.~\eqref{eq:subspace_V}.

Then for any $\lambda \in \sigma(A,M) \cap (0,\Lambda)$ there exists $\tilde \lambda \in \sigma_V(A,M)$ such that 
\begin{equation} 
\label{FinalRelativeErrorEstimate}
\frac{ | \lambda - \tilde{\lambda} |}{\lambda} 
\leq C_M C(\lambda) \Lambda (4\gamma)^3  \, \left (\frac{1}{4 (\gamma -  1)} \right )^{2N+2} \quad \text{for any} \quad \gamma > 1
\end{equation}
where $C_M,C(\lambda)$ are positive constants defined in Lemma
\ref{RestrictionEstimateLemma} and Proposition \ref{prop:eigen_proj_estimate},
respectively.
\end{theorem}
\noindent For a given $\Lambda$, parameters $\gamma,N$ determine both the accuracy and the computational cost of the CPI method. Specifically, all exterior eigenpairs $(\mu_k,\vec{v}_k)$ satisfying $\mu_k \leq \gamma \Lambda$ together with $n_\Gamma N$ linear systems have to be solved. The spectra of $\sigma(A,M)$ and $\sigma(A_{22},M_{22})$ do not restrict the choice of $\gamma$. 
\begin{proof} The claim follows by estimating the coefficients
  $\beta_k(\lambda)$ and $c_k(\lambda)$ in Eq.~\eqref{eq:def_bk_ck} using
  Lemma~\ref{lemma:help1}. Estimate for $c_k(\lambda)$ is given in
  Eq.~\eqref{LagrangeErrorEstimates}. We proceed to estimate the coefficients
  $\beta_k(\lambda)$. For $\lambda \in \sigma(A,M) \cap (0,\tilde{\Lambda} )$ Eqs.~\eqref{eq:smap}~and~\eqref{eq:Zexp} yield
\begin{equation*}
(I-P_{\tilde \Lambda}) \vec{x}_2 = f_{\tilde \Lambda} (\lambda) Z(\lambda) \vec{x}_1 =  \sum_{\substack{k \\ \mu_k > \tilde \Lambda}} \frac{\beta_k(\lambda)}{\mu_k - \lambda} \vec{v}_k.
\end{equation*}
By Lemma~\ref{RestrictionEstimateLemma}, properties of $P_{\tilde \Lambda}$, and normalisation of the eigenvectors of
problem \eqref{eq:AlgEigValP}, we have
\begin{equation} \label{BetaEstimates}
  C_{M}  = C_{M}  \vec{x}^T M \vec{x} \geq \| \vec{x} \|_{M_{22}}^2
  \geq  \sum_{\substack{k \\ \mu_k > \tilde \Lambda}} \left( \frac{\beta_k(\lambda)}{\mu_k - \lambda} \right)^2.
\end{equation}
Combining these with Eq.~\eqref{LagrangeErrorEstimates} gives the estimate
\begin{equation*}
  \begin{aligned}
    & \sum_{\substack{k \\ \mu_k > \tilde \Lambda}} \mu_k c_k(\lambda)^2 \beta_k(\lambda)^2 \\
    & \leq \sum_{\substack{k \\ \mu_k > \tilde \Lambda}} 
    \mu_k  \left ( \frac{ \beta_k(\lambda)}{\mu_k- \lambda } \right )^2 
    \frac{\Lambda^{2N}(\mu_k- \lambda)^2}{4^{2N-1}(\mu_k - \Lambda)^{2N+2}}  \\
    &
    \leq  4 \left( \frac{\Lambda}{4} \right)^{2N} \max_{\substack{k \\ \mu_k > \tilde \Lambda}}
    \left (  \frac{\mu_k (\mu_k- \lambda)^2}{(\mu_k - \Lambda)^{2N+2}}  \right ) 
    \sum_{\substack{k \\ \mu_k > \tilde \Lambda}} \left ( \frac{  \beta_k(\lambda)}{\mu_k- \lambda} \right )^2  \\
    & \leq 4 C_M \left (\frac{\Lambda}{4} \right )^{2N} \frac{\tilde \Lambda  (\tilde \Lambda - \lambda)^2}{(\tilde \Lambda-\Lambda)^{2N+2}} 
  \leq \frac{4  C_M \tilde \Lambda  (\tilde \Lambda - \lambda)^2}{(\tilde \Lambda - \Lambda)^{2}}  \cdot \left (\frac{L}{4 (\tilde \Lambda - \Lambda)} \right )^{2N} \\
&  \leq \frac{4  C_M (\gamma \Lambda - \lambda )}{(1 - 1/\gamma)^{2}}  \cdot \left (\frac{1}{4 (\gamma -  1)} \right )^{2N} 
= \frac{4 C_M \gamma^2 (\gamma \Lambda - \lambda )}{(\gamma - 1)^{2}}  \cdot \left (\frac{1}{4 (\gamma -  1)} \right )^{2N}.
  \end{aligned}
\end{equation*}
We used here the fact that the function
\begin{equation*}
x \mapsto \frac{x(x-\lambda)^2}{(x- \Lambda)^{2N + 2}} \quad \text{for} \quad x > 0, \quad x \neq \Lambda
\end{equation*}
is decreasing for $x > \Lambda$, and hence its maximum over
$[\tilde \Lambda,\infty)$ is attained at $x = \tilde \Lambda$. Finally, we
use $\lambda > 0$ to obtain the final estimate
\begin{equation*}
\frac{ | \lambda - \tilde{\lambda} |}{\lambda} 
\leq \frac{4 C_M  C(\lambda) \Lambda \gamma^3 }{(\gamma - 1)^2}  \, \left (\frac{1}{4 (\gamma -  1)} \right )^{2N} 
= C_M C(\lambda) \Lambda (4\gamma)^3   \, \left (\frac{1}{4 (\gamma -  1)} \right )^{2N+2}. 
\end{equation*}
This completes the proof.
\end{proof}

If we fix $\gamma > 5/4$ in Theorem \ref{th:FinalEstimate}, the right hand side
of Eq.~\eqref{FinalRelativeErrorEstimate} converges to zero as $N \to \infty$.
Observe that the set $\{ k \,: \, \mu_k > \gamma \Lambda \} = \emptyset$ for
$\gamma$ large enough for any $\Lambda > 0$. Then the sum in
inequality~\eqref{BetaEstimates} vanishes, and again $\lambda = \tilde \lambda$
follows.

\section{Computational cost}
\label{sec:cost}
The error estimate given in Theorem~\ref{th:FinalEstimate} allows one to choose
the values for $N$ and~$\gamma$ in an optimal way, depending on the target error
level and a model for the computational cost required to solve the eigenvalue
problem. Solving the smallest elements $\lambda \in \sigma_V(A,M)$ using the
method matrix $Q$ (as given in Eq.~\eqref{eq:ourmethodmatrix}) amounts to
solving $\hat A \vec{x} = \lambda \hat M \vec{x}$ in which
\begin{equation*}
\hat{A} = \begin{bmatrix}
A_{11} & A_{12} Q_{22} \\ 
Q_{22}^T A_{12}^T & Q_{22}^T A_{22} Q_{22}
\end{bmatrix}\quad \textrm{ and }
\hat{M} = \begin{bmatrix}
M_{11} & M_{12} Q_{22} \\ 
Q_{22}^T M_{12}^T & Q_{22}^T M_{22} Q_{22}
\end{bmatrix}.
\end{equation*}
By Section~\ref{sec:mmatrix},  $Q_{22} \in \mathbb{R}^{n_2 \times \mathop{dim}(V_2)}$, and hence
\begin{equation*}
Q_{22}^T A_{22} Q_{22} \in  \mathbb{R}^{\mathop{dim}(V_2) \times \mathop{dim}(V_2)} \text{ where } \mathop{dim}(V_2) \leq K(\gamma \Lambda) + n_\Gamma N.
\end{equation*} 

Denote the Cholesky factorisations of the matrices $\hat A$, $\hat M$ by $\hat A = R^T R$ and $\hat M = L^TL$, respectively.  Eigenvalues $\tilde \lambda \in \sigma_{V}(A,M) \cap (0,\Lambda)$ can be solved, e.g., by applying the Lanczos iteration to  
\begin{equation*}
L \hat A^{-1} L^T \vec{y} = \tilde \lambda^{-1} \vec{y} \quad \text{where} \quad \hat M = L^TL \quad \text{and} \quad \vec{y} = L \vec{x}.
\end{equation*}
This requires repeated multiplications by $L \hat A^{-1} L^T$ which can be
efficiently carried out using Cholesky factorisations. The factorisations should
be computed by taking advantage of the block structure: for example, by writing
$\hat{A} = R^T R$ so that
\begin{equation}
\label{eq:blockChol}
 R = \begin{bmatrix} R_{11} & R_{11}^{-T} \hat{A}_{12} \\ 0 & R_{22} \end{bmatrix}, \hat{A}_{11} = R_{11}^T R_{11}, \text{and } \hat{A}_{22} - \hat A_{12}^T \hat{A}_{11}^{-1} \hat A_{12} = R_{22}^T R_{22}
\end{equation}
where the Schur complement of $\hat A$ with respect to $A_{11}$ has also been Cholesky factorised. A similar formula can be used for $\hat{M}$ to produce $L$ in block form.  The matrix $R_{22}$ can be computed as a low-rank update to the factorisation of $\hat A_{22}$. This leads to 
\begin{equation*}
L R^{-1} R^{-T} L^T \vec{y} = \tilde \lambda^{-1} \vec{y}.
\end{equation*}
The cost of the matrix-vector multiplication by~$L R^{-1} R^{-T} L^T$ is of
lower order compared to computing the factorisations which we discuss next. The block structure can be used in the spirit of Eq.~\eqref{eq:blockChol} to recycle the factorisation of $A_{22}$ in the shift-and-invert Lanczos iteration.

For each version of problem~\eqref{eq:AlgEigValP}, one has to recompute the
Cholesky factorisations in Eq. \eqref{eq:blockChol}. The cost of factorising
$A_{11}$ does not depend on the choice of $V_2$. Hence, we only model the cost
of computing the Cholesky factorisation for the Schur complement $\hat{A}_{22} -
\hat A_{12}^T \hat{A}_{11}^{-1} \hat A_{12}$. Depending on the underlying
problem, the Schur complement can be sparse or dense. Thus, the cost of
computing the factorisation is modelled as proportional to the $r$th power of
$\mathop{dim}(V_2)$ as
\begin{equation}
\label{eq:cost}
\mathop{cost}(\gamma, N) := \left( K(\gamma \Lambda) + n_{\Gamma} N \right)^r
\end{equation}
where $K(\ell)$ is as defined in Eq.~\eqref{eq:ngKl} and the parameter $r \in
[0,3]$ depends on the sparsity of the Schur complement.

\subsection{Optimisation of $N$ and $\gamma$}

A typical application for CPI is the solution of eigenvalues for the Dirichlet
Laplacian in $\Omega \subset \mathbb{R}^d$ using the finite element method. In
this case, an asymptotically accurate description for $K(l)$ is given by the
Weyl law~\cite[Ch. 8]{Roe:Elliptic:1988} as $K( l ) \approx C(d)
\mathop{vol}(\Omega_2) l^{d/2}$ where $C(d) = (2\pi)^{-d} \mathop{vol}B_d$ and
$B_d$ is a $d$-dimensional unit ball. Motivated by Theorem
\ref{th:FinalEstimate}, we define a normalised tolerance function:
\begin{equation}
\label{eq:tol}
\mathop{ntol}(\gamma,N) := \gamma^3 \left( \frac{1}{4(\gamma-1)} \right)^{2N+2}.
\end{equation} 
When $\gamma$ and $N$ are chosen such that $\mathop{ntol}(\gamma,N) \leq \eta$,
the relative error in eigenvalues $\lambda \in (0,\Lambda)$ satisfies
\begin{equation*}
\frac{|\lambda - \tilde \lambda |}{\lambda} \leq 64 C_M \Lambda C(\lambda) \eta
\end{equation*}
 by Theorem \ref{th:FinalEstimate}.
 An optimal value combination for parameters $N$ and $\gamma$ for a normalised
 error level $\eta$ is obtained by minimising the cost function \eqref{eq:cost}
 under the constraint $ntol(\gamma,N) = \eta$.
\begin{theorem}
\label{th:optimalparameters}

Let $\eta>0$, and let $ntol(\gamma,N)$ be as given in Eq.~\eqref{eq:tol}.
Define the computational cost model function as
\begin{equation*}
g(\gamma,N) := \left( C(d) \mathop{vol}(\Omega_2) ( \gamma \Lambda)^{d/2} +  n_\Gamma N \right)^r
\end{equation*}
where $C(d) = (2\pi)^{-d} \mathop{vol}B_d$, and $B_d$ is a $d$-dimensional unit
ball. Then the computational cost model function is minimised under the
constraint $ntol(N,\gamma) = \eta$ by choosing $N(\gamma)$ and $\gamma$ such
that
\begin{equation}
\label{eq:gamma2N}
  N(\gamma) =  \frac{d C(d) \Lambda^{d/2} \mathop{vol}(\Omega_2)}{2n_\Gamma} \gamma^{d/2-1}(\gamma - 1) \ln{4(\gamma - 1)} - 2.
\end{equation}
and
\begin{equation}
\label{eq:tol2gamma}
\gamma^3 \left( \frac{1}{4(\gamma-1)} \right)^{2N(\gamma)+2} = \eta.
\end{equation}
\end{theorem}
\noindent This theorem follows by minimising $C(d) \mathop{vol}(\Omega_2) (
\gamma \Lambda)^{d/2} + n_\Gamma N$ under the constraint $ntol(\gamma,N) = \eta$
by using the method of Lagrange multipliers. Observe that
Eqs.~\eqref{eq:gamma2N}~and~\eqref{eq:tol2gamma} do not depend on the exponent
$r$ in the approximate cost function $g(\gamma,N)$.

\begin{remark} From Eqs.~\eqref{eq:gamma2N}~and~\eqref{eq:tol2gamma} one can
  numerically solve $N$ and $\gamma$ as a function of $\eta$. The resulting $N$
  is typically not an integer but it can be rounded up while preserving the
  desired normalised tolerance $\eta$. For $d=2$, the graphical approximation
  given in Fig.~\ref{fig:nomogram}~can be used. Denote
\begin{equation*}
\tilde \eta := \eta^{\frac{2\pi n_{\Gamma}}{\Lambda}} \quad \textrm{ and } \quad \tilde{N}(\gamma) := ( \gamma-1) \ln{ 4(\gamma-1)}.
\end{equation*}
Since $27/64 \leq \gamma^{3}(4 ( \gamma - 1) )^{-2} \leq 1/2$ for $\gamma \in
[2,5]$ we have 
\begin{equation*}
\tilde \eta(\gamma) \approx 2 \left( \frac{1}{4 ( \gamma - 1) }\right)^{\tilde{N}(\gamma)}.
\end{equation*}
Using this approximation to determine $\gamma$ eliminates $n_\Gamma$ and
$\Lambda$ from the graphical procedure. The value for $N(\gamma)$ is recovered
from $N(\gamma) = \frac{\Lambda }{2\pi n_\Gamma} \tilde{N}(\gamma)$ and rounded
up.
\begin{figure}[h]
  \centering
  \includegraphics[width=.6\textwidth]{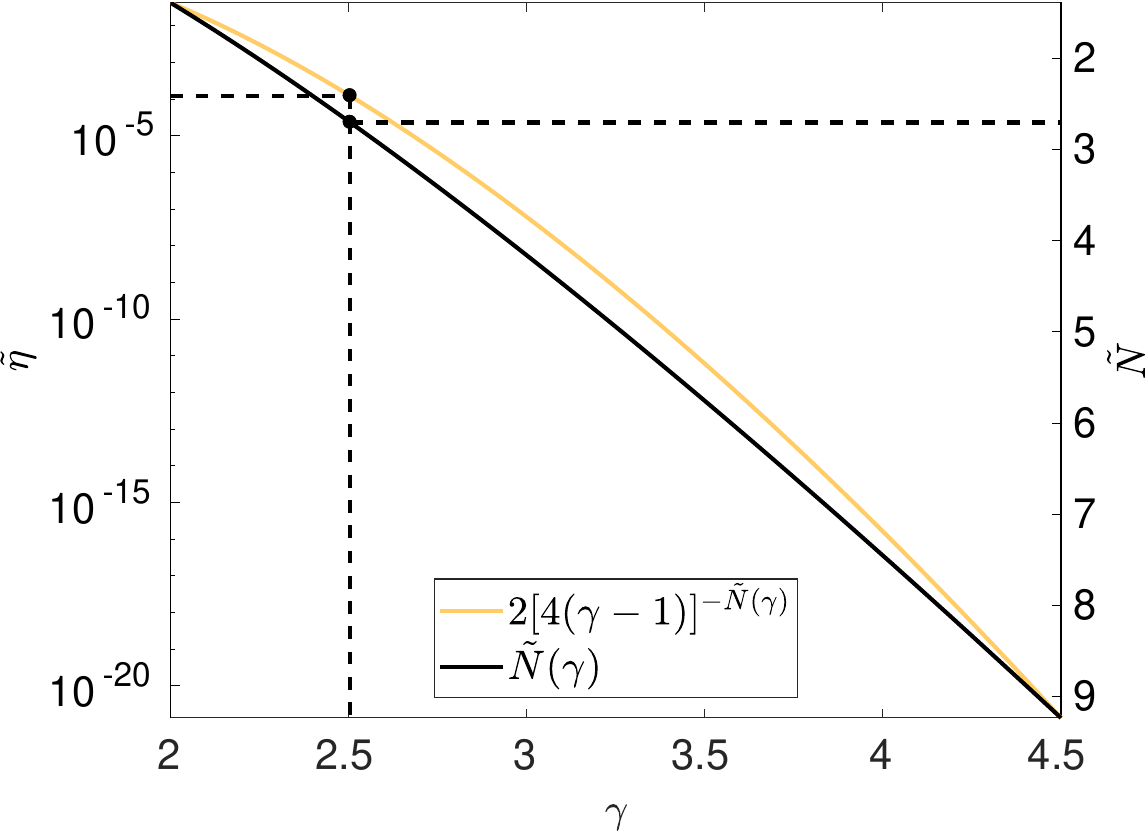}
  \caption{Graphical tool for finding the auxiliary parameter value
    $\tilde{N}(\gamma)$ for a given $\tilde{\eta}$ in the case $d=2$ for the Laplace equation.}
  \label{fig:nomogram}
\end{figure}
\end{remark}

\section{Dimension reduction}
\label{sec:dim_red}
It usually turns out that the space $V_2$ is
excessively large for the chosen error level, and it can be replaced
by $\tilde{V}_2 \subset V_2$ of considerably smaller dimension while
maintaining desired accuracy. We proceed to discuss how such
$\tilde{V}_2$ can be constructed.

For ease of presentation, assume now that the eigenvalues $\{ \mu_k \}$ in
Eq.~\eqref{ExteriorEVP} are given in non-decreasing order, and write $K =
K(\tilde \Lambda) $ as in Eq.~\eqref{eq:ngKl}. Recall the definition of $B$ from
Section~\ref{sec:mmatrix}
\begin{equation}
  \label{eq:defB}
  B =  \begin{bmatrix} \vec{v}_1,\;  \ldots  ,\; \vec{v}_K,\; (I-P_{\tilde
  \Lambda}) \vec q_{11},\;  \ldots ,\; (I-P_{\tilde
  \Lambda}) \vec q_{Nr} \end{bmatrix}
\end{equation}
which satisfies $V_2 = \mathop{range}(B)$. Note that the matrix $B$ may
have a non-trivial null space. Hence, the $Q_{22}$-block in the method matrix
$Q$ defined by Eq.~\eqref{eq:ourmethodmatrix} is obtained by computing a basis
for $\mathop{range}(B)$ using, e.g., SVD of $B$.

Denoting a low-rank approximation of $B$ by $\tilde
B$ with $\tilde{V}_2 := \mathop{range}(\tilde B)$, the corresponding method
matrix for the \emph{dimension reduced version of CPI} is given by
\begin{equation} \label{eq:redSubSp}
\tilde Q := \begin{bmatrix} I & 0 \\ 0 & \tilde Q_{22} \end{bmatrix} 
\textrm{ where } \mathop{ker}(\tilde Q_{22}) = \{0\} \textrm{ and } \mathop{range}(\tilde Q_{22}) = \mathop{range}(\tilde B).
\end{equation}
Further, let $\tilde V := \mathop{range}(\tilde Q)$. We proceed to give an error estimate for the dimension reduced version of CPI.
\begin{lemma}  \label{DimRedLemma}
Let $(\lambda,\vec{x}) \in (0,\Lambda) \times \mathbb{R}^{n} \setminus
\{ 0 \}$ be an eigenpair of Eq.~\eqref{eq:AlgEigValP} with
$\vec{x} = \begin{bmatrix} \vec{x}_1 & \vec{x}_2 \end{bmatrix}^T$
according to the standard splitting. Let $\{ \vec p_1, \ldots \vec p_r
\} \subset \mathbb{R}^{n_2}$ be vectors satisfying
Eq.~\eqref{pvectors}.  Let $B \in \mathbb{R}^{n_2 \times (K+Nr)}$
be as defined in Eq.~\eqref{eq:defB} and
\begin{equation} \label{eq:explicit_u2}
  \vec{u}_2 := P_{\tilde{\Lambda}} \vec{x}_2 +\sum_{i=1}^N \ell_i(\lambda) f_{\tilde \Lambda}(\xi_i) Z(\lambda) \vec x_1
\quad \text{ where } \quad Z(\lambda) = \lambda M_{21} - A_{21}.
\end{equation}
Define the dimension reduced method matrix and the corresponding
subspace by Eq.~\eqref{eq:redSubSp}.
Then there exists $\tilde{\lambda} \in \sigma_{\tilde V}(A,M)$ such
that
\begin{equation}
 \label{eq:bestapprox_tol} 
 \frac{| \lambda - \tilde{\lambda} |}{\lambda} \leq
 2 C_M C(\lambda) \Lambda  (4\gamma)^3   \, \left (\frac{1}{4 (\gamma -  1)} \right )^{2N+2} + 
 2 \left \| R \left ( B - \tilde B \right ) \right \|^2 \min_{B  
   \hat{\vec{\alpha}} = \vec{u}_2}{\left \| \hat{\vec{\alpha}} \right \|^2}
\end{equation} 
where $A_{22} = R^T R$ is the Cholesky factorisation of $A_{22}$, and
the constants $N,\gamma,C_M$, and $C(\lambda)$ are as in
Theorem~\ref{th:FinalEstimate}.
\end{lemma}
\noindent In computations, one would choose an optimal combination of $N$
and $\gamma$ as described in Section~\ref{sec:cost} for the
untruncated version of CPI.
\begin{proof}
The original error estimate~\eqref{FinalRelativeErrorEstimate} in
Theorem~\ref{th:FinalEstimate} was derived by implicitly constructing
$\vec{u}_2$ in Eqs.~\eqref{eq:v2red1}~and~\eqref{eq:w2def} in order to
bound the right hand side of Eq.~\eqref{eq:rel_error}, i.e.,
\begin{equation*}
\min_{\vec{v}_2 \in V_2} \| \vec{x}_2 - \vec{v}_2 \|_{A_{22}} \leq  \| \vec{x}_2 - \vec{u}_2 \|_{A_{22}}.
\end{equation*}
In that theorem, the latter term is further estimated by
\begin{equation*}
\| \vec{x}_2 - \vec{u}_2 \|_{A_{22}}^2 \leq C_M \Lambda (4\gamma)^3  \, \left (\frac{1}{4 (\gamma -  1)} \right )^{2N+2}.
\end{equation*}
The proof of the current claim follows from this by a perturbation
argument. Let $\vec \alpha$ be such that $\vec{u}_2 = B \vec
\alpha$.  Since $\tilde{V}_2 = \mathop{range}(\tilde B)$, we have
\begin{equation*}
\begin{aligned}
  &  \min_{\vec{v}_2 \in \tilde{V}_2} \| \vec{x}_2 - \vec{v}_2 \|_{A_{22}}
  \leq \| \vec{x}_2 - \tilde{B} \vec{\alpha} \|_{A_{22}} 
  = \| \vec{x}_2 - B \vec{\alpha} +  (B - \tilde{B}) \vec{\alpha}  \|_{A_{22}} \\
  & \leq \|
  \vec{x}_2 - \vec{u}_2 \|_{A_{22}} + \| (B - \tilde{B})  
  \vec{\alpha} \|_{A_{22}}
\leq \|
  \vec{x}_2 - \vec{u}_2 \|_{A_{22}} + \| R (B - \tilde{B}) \|  \|
  \vec{\alpha} \|.
\end{aligned}
\end{equation*}
The claim follows by squaring this estimate and applying
Proposition~\ref{prop:eigen_proj_estimate} with $\tilde V$ in place of
$V$.
\end{proof}

To make practical use of Lemma~\ref{DimRedLemma} to achieve a given
target level for the relative eigenvalue error, we start by bounding
the first term in Eq.~\eqref{eq:bestapprox_tol} by  choosing the
parameter value combination $N$, $\gamma$ using Eqs.~\eqref{eq:gamma2N}
and~\eqref{eq:tol2gamma}. It remains to bound the second term in
Eq.~\eqref{eq:bestapprox_tol} so that
\begin{equation*}
  \left \| R \left ( B - \tilde{B} \right ) \right \|
  \min_{ B  \hat{\vec{\alpha}} = \vec{u}_2}{\left \| \hat{\vec{\alpha}} \right \|}
\leq \| R ( B - \tilde{B} ) \| \| \vec{\alpha} \| < \sqrt{tol}
\end{equation*}
for a given truncation error level $tol > 0$. A vector $\vec{\alpha}$ satisfying $Q_{22} \vec{\alpha} = \vec{u}_2$ and upper bound for $\| \vec{\alpha} \|$ are given below. Given such $\vec{\alpha}$, 
we then use the SVD $R B = \sum_{i=1}^{n_2} \sigma_i \vec{u}_i
\vec{w}^T_i$ (where $\sigma_i$~are ordered in non-increasing order) to
construct the lowest rank $\tilde{B}$  satisfying $\| R ( B - \tilde{B} ) \| \| \vec{\alpha} \| < \sqrt{tol}$ as
\begin{equation*}
\tilde B = R^{-1} \sum_{i=1}^{K_c} \sigma_i \vec{u}_i \vec{w}^T_i \quad \text{ where  } \quad K_c := \max \{ i \; | \sigma_i 
> \sqrt{tol} /  \|  \vec{\alpha} \|  \}.
\end{equation*}
The method matrix block $\tilde{Q}_{22}$ for the dimension reduced CPI, defined in Eq.~\eqref{eq:redSubSp}, is obtained as 
\begin{equation} \label{eq:Bcut}
\tilde{Q}_{22} :=   
R^{-1} \begin{bmatrix} \vec{u}_1 & \vec{u}_2 & \ldots & \vec{u}_{K_c}  \end{bmatrix}
\end{equation}
where the column vectors of $\tilde{Q}_{22}$ are orthonormal in the $A_{22}$-inner product. Because $\vec \alpha \neq \vec{0}$, the number $K_c = K_c(\left \|
\vec{\alpha} \right \|, tol)$ is always defined, and the truncation
error level $tol > 0$ can be chosen arbitrarily small. 

We make use of the Lebesgue constant $\Lambda_N$ for Chebyshev
interpolation points of $(0,\tilde \Lambda)$ (see, e.g.,
~\cite{Brutman:1978}), given by
\begin{equation} \label{eq:Lebesgue}
  \Lambda_N := \max_{ t \in (0,\tilde \Lambda)}  \sum_{i=1}^N \left | \ell_i(t) \right | 
  \quad \textrm{ that satisfy } \quad \Lambda_N =  \frac{2}{\pi} \log{N} + \mathcal{O}(1).
\end{equation}
\begin{lemma} \label{alphaLemma}
Make the same assumptions and use the same notation as in
  Lemma~\ref{DimRedLemma}.
Then there exists $\vec \alpha$ satisfying $\vec{u}_2 = Q_{22}
\vec{\alpha}$, such that
\begin{equation}
\label{eq:alpha_bound1}
\| \vec{\alpha} \|^2 \leq 1 + \Lambda_N^2 \| \vec{\theta} \|^2 
\end{equation} 
where the coefficient vector  $\vec{\theta} := \begin{bmatrix} \theta_1 & \theta_2 & \ldots & \theta_r 
\end{bmatrix}^T$ satisfies
\begin{equation} \label{eq:alpha}
Z(\lambda) \vec{x}_1 = \sum_{j=1}^r \theta_j \vec{p}_j.
\end{equation}
\end{lemma}
\begin{proof} 
Define the coefficients $\tau_k$ by the expansion
\begin{equation} 
  \label{eq:gamma}
  P_{\tilde \Lambda} \vec{x}_2 = \sum_{k=1}^K \tau_k \vec v_k \quad  \textrm{ so that  } 
  \quad  \| P_{\tilde \Lambda} \vec{x}_2 \|^2_{M_{22}} = \sum_{k=1}^K \tau_k^2 \leq 1.
\end{equation}
Indeed, this holds by the $M_{22}$-orthogonality of $\{ \vec v_k \}$
(see Eq.~\eqref{eq:Pdef}) and the normalisation $\vec{x}^T M \vec{x} =
1$ implying $\| \vec x_2 \|_{M_{22}} \leq 1$ and hence $\| P_{\tilde
  \Lambda } \vec{x}_2\|_{M_{22}} \leq 1$.

Define $\alpha_{i,j} := \ell_i(\lambda) \theta_j$ where $i=1,\ldots N$
and $i=j,\ldots r$. Then $\vec{u}_2$ in Eq.~\eqref{eq:explicit_u2} can
be written in the form
\begin{equation*}
\vec{u}_2 = Q_{22} \vec{ \alpha } = \sum_{k=1}^K \tau_k \vec{v}_k +
\sum_{i=1}^N \sum_{j=1}^r \alpha_{i,j} (I-P_{\tilde \Lambda}) \vec{q}_{i,j} 
\end{equation*}
by Eqs.~\eqref{SampleVectors}--\eqref{eq:defB}.  So, we can choose
\begin{equation*}
\vec{\alpha} :=
  \begin{bmatrix} \tau_1 & \ldots & \tau_K & \alpha_{1,1} & \ldots & \alpha_{N,1} & \ldots & \alpha_{1,r} & \ldots & \alpha_{N,r} \end{bmatrix}^T.
\end{equation*}
By Eq.~\eqref{eq:gamma} and the definition of $\alpha_{i,j}$ we have
\begin{equation*}
\| \vec{\alpha} \|^2 = \sum_{k=1}^K \tau_k^2 + \sum_{i=1}^N
\sum_{j=1}^r \alpha_{i,j}^2 \leq 1 + \sum_{i=1}^N \ell_i(\lambda)^2
\sum_{j=1}^r \theta_j^2.
\end{equation*} 
Observing that $\sum_{i=1}^N \ell_i(\lambda)^2 \leq \Lambda_N^2$
completes the proof.
\end{proof}

The magnitude of $\| \vec{\theta} \|$ in Eq.~\eqref{eq:alpha_bound1} is
estimated by choosing the vectors $\{ \vec p_1, \ldots , \vec{p}_r \}$
in a way that $\theta_j$'s can be explicitly solved from the
Eq.~\eqref{eq:alpha}:
\begin{lemma} \label{alphaCorollary}
  Make the same assumptions and use the same notation as in
  Lemma~\ref{DimRedLemma}.  Assume that the matrices $A_{21}$,
  $M_{21}$ are ordered so that only their first $r_1$, $r_2$ column
  vectors are nonvanishing, respectively.
Let
\begin{equation}
\label{eq:pset}
\{ \vec p_1, \ldots , \vec{p}_r \} = \{ A_{21} \vec e_1, \ldots, A_{21}  \vec e_{r_1} ,  M_{21}  \vec e_1, \ldots,  M_{21} \vec e_{r_2} \}.
\end{equation}
Then there exists a coefficient vector $\vec{\theta}
:= \begin{bmatrix} \theta_1 & \theta_2 & \ldots & \theta_r
\end{bmatrix}^T$  such that $Z(\lambda) \vec{x}_1 =
\sum_{j=1}^r \theta_j \vec p_j$ and
\begin{equation*}
\| \vec \theta\|^2 \leq  (1 + \lambda^2) \|M^{-1}\|. 
\end{equation*}
\end{lemma}

\begin{proof} Since $Z(\lambda) = \lambda M_{21} - A_{21}$ holds,   it follows 
from Eq.~\eqref{eq:pset} that
\begin{equation*}
\label{eq:ZLx1}
\begin{aligned}
 Z(\lambda) \vec x_1 & = \sum_{j=1}^{r_1} \lambda x_{1,j}  M_{21} \vec{e}_j - \sum_{j=1}^{r_2} x_{1,j}  A_{21} \vec{e}_j \text{ and }\\
 \sum_{j=1}^r \theta_j \vec p_j & = \sum_{j=1}^{r_1} \theta_j M_{21} \vec{e}_j + \sum_{j=r_1+1}^{r_2} \theta_j  A_{21} \vec{e}_j.
\end{aligned}
\end{equation*}
Hence, one solution of Eq.~\eqref{eq:alpha} is 
$\vec \theta = \begin{bmatrix}
\lambda x_{1,1} & \ldots & \lambda x_{1,r_1} & -x_{1,1} & \ldots & -x_{1,r_2}
\end{bmatrix}$, and it satisfies the estimate
\begin{equation}
\label{eq:t_est}
  \begin{aligned}
    \| \vec \theta\|^2 & =  \lambda^2 \sum_{j=1}^{r_1} x_{1,j}^2 +  \sum_{j=1}^{r_2} x_{1,j}^2
     \leq \left (1+\lambda^2 \right)  \| \vec{x}_1 \|^2 \\
     &  \leq \left (1+\lambda^2 \right)  \| \vec{x} \|^2
     =  \left (1+\lambda^2 \right)  \frac{\vec{x}^T \vec{x}}  {\vec{x}^T M \vec{x}} \\
     &  \leq \left (1+\lambda^2 \right) \max_{\vec{v} \in \mathbb{R}^n ,
  \vec{v} \neq 0} \frac{ \vec{v}^T \vec{v}}{ \vec{v}^T M \vec{v}}  
     = \left (1 + \lambda^2 \right ) \|M^{-1}\|
\end{aligned}
\end{equation}
because $\vec{x}^T M \vec{x}= 1$. The proof is now complete.
\end{proof}

\noindent The combination of Lemmas~\ref{DimRedLemma},~\ref{alphaLemma},~and~\ref{alphaCorollary}
yields the following result:
\begin{theorem}
  \label{thm:alpha_estimate}
Let $(\lambda,\vec{x}) \in (0,\Lambda) \times \mathbb{R}^{n} \setminus
\{ 0 \}$ be an eigenpair of Eq.~\eqref{eq:AlgEigValP}.
Let the vectors $\{ \vec p_1, \ldots , \vec{p}_r \}$ be defined by
Eq.~\eqref{eq:pset} and the matrix $Q_{22} \in \mathbb{R}^{n_2 \times
  (K + Nr)}$ as in Eq.~\eqref{eq:defB}. Define $R$ by the Cholesky factorisation $A_{22} = R^T R  $, and let
$\sum_{i=1}^{n_2} \sigma_i \vec{u}_i \vec{w}_i^T$ be the SVD of
$RQ_{22}$. For any truncation error level $tol > 0$, define 
  \begin{equation*}
    K_c(tol) := \max \{\, i \,|\, \sigma_i^2 \left( 1 + \Lambda_N^2 \left (1 + \lambda^2 \right )\|M^{-1}\|  \right) >  tol   \,\},
  \end{equation*}
where $\Lambda_N$ is given by Eq.~\eqref{eq:Lebesgue}. Define the
method matrix $\tilde{Q}$ and the subspace $\tilde V$ by
Eqs.~\eqref{eq:redSubSp}~and~\eqref{eq:Bcut}.

Then there exists $\tilde{\lambda} \in \sigma_{\tilde{V}}(A,M)$ such
that
  \begin{equation*}
    \frac{|\lambda - \tilde{\lambda}|}{\lambda} \leq 2 C_M C(\lambda) \Lambda (4\gamma)^3 \left( \frac{1}{4(\gamma-1)} \right)^{2 N + 2} + 2 tol
  \end{equation*}
for any parameter value combination $N, \gamma$ where the constants $C_M$
and $C(\lambda)$ are as in Theorem~\ref{th:FinalEstimate}.
\end{theorem}

A typical application of CPI is the solution of the lowest
eigenmodes of the Laplace operator in a bounded domain $\Omega \subset \mathbb{R}^d$
using the finite element method. When piecewise linear basis functions are used on quasi-uniform simplicial meshes, it is known that $\| M^{-1} \| \leq Ch^{-d}$ where the constant $C$ is independent on the mesh size $h$, see~\cite[Section~6]{Braess:2007}. A reasonable value for the cut-off index in Theorem~\ref{thm:alpha_estimate}~can be computed with the help of this estimate.

\section{Model problems}
\label{sec:num_ex}

We proceed to illustrate theoretical results by two numerical
examples. Both examples involve the eigenpairs $(\lambda',u) \in
(0,\Lambda) \times \mathcal{V}$ of the variational eigenvalue problem 
\begin{equation} \label{eq:cont_eigen_again}
(\nabla u,\nabla v) = \lambda' (u,v) \quad \text{for all } \quad v \in \mathcal{V}
\end{equation}
of the Laplace operator where $\mathcal{V} \subset
H^1(\Omega)$ is a subspace where the homogeneous Dirichlet boundary
condition holds at least on a part of the boundary $\partial
\Omega$. Problem~\eqref{eq:cont_eigen_again} is discretised using
finite element method with piecewise linear basis functions
leading to the algebraic eigenvalue problem \eqref{eq:AlgEigValP} that
is the subject matter of this article.

\subsection{Computational considerations}

Let us begin by describing an implementation of CPI. The
problem data consists of the spectral interval of interest
$(0,\Lambda)$, the specified upper bound for the relative eigenvalue
error, and the symmetric positive definite stiffness and mass matrices
$A$ and $M$.  Without loss of generality, the basis functions can be
assumed to be ordered so that $A$ and $M$ obey the standard splitting
given in Eq.~\eqref{eq:dec_sys} corresponding to the interior and the
exterior systems. The purpose is to compute spectral approximations
\begin{equation*}
\sigma_V(A,M) \cap (0,\Lambda) \approx \sigma(A,M) \cap (0,\Lambda)
\end{equation*}
for several versions of Eq.~\eqref{eq:AlgEigValP} sharing the same
exterior system. Note that the dimension of the eigenvalue problem may vary 
between different versions as long as $\mathop{range}(A_{21})$ and 
$\mathop{range}(M_{21})$ remain fixed. Thus, the finite element mesh 
of the exterior part stays constant while mesh of the interior part may
vary.

As discussed in Section~\ref{sec:cost}, an effective choice of 
$\gamma$ and $N$ requires \emph{a~priori} information on the eigenvalue
distribution of problem~\eqref{eq:cont_eigen_again}~that is encoded in the 
function $K(l)$ in Eq.~\eqref{eq:ngKl}. We model $K(l)$
by the Weyl law as in Section~\ref{sec:cost}.  Values for $\gamma$ and $N$
are then chosen using Theorem~\ref{th:optimalparameters}. For practical reasons, 
we set $C(\lambda) = C_M = 1$ in Eq.~\eqref{FinalRelativeErrorEstimate}. For 
a given finite element mesh size $h>0$, the term $\| \vec{\alpha} \|$ is 
approximated by setting $\|M^{-1} \| = h^{-d}$ and applying Lemmas
\ref{alphaLemma} and \ref{alphaCorollary}.

CPI consists of the following steps:


\begin{enumerate}
\item Using the target relative error level, determine $N$ and $\gamma$ using,
  e.g., Eqs.~\eqref{eq:gamma2N} and~\eqref{eq:tol2gamma}. Compute the Chebyshev
  points $\{ \xi_i \}_{i=1}^N$ using Eq.~\eqref{eq:defChebyIP}.
  
\item Compute the $K = K(\gamma \Lambda)$ smallest eigenpairs $(\mu_k,
  \vec{v}_k)$ of the exterior system $A_{22} \vec v_k = \mu_k M_{22}
  \vec v_k$.

\item Let $\vec{p}_j$, $j=1,\ldots,r_1$, and $\vec{p}_j$,
  $j=r_1+1,\ldots,r$ with $r = r_1+r_2$, be the nonzero columns of
  $M_{21}$ and $A_{21}$, respectively. Compute the sample vectors
  $\vec q_{11},\ldots,\vec q_{Nr}$ as solutions of $(A_{22} - \xi_i
  M_{22}) \vec{q}_{i j} = \vec p_j$.
  
\item Collect eigenvectors from Step~1 and sample vectors from Step~2 into
  matrix $B = [\vec v_1, \ldots, \vec v_K, (I-P_{\tilde{\Lambda}}) \vec q_{11}, \ldots, (I-P_{\tilde{\Lambda}}) \vec q_{Nr}]$.
  Compute the SVD $RB = \sum_i \sigma_i \vec{u}_i \vec{w}^T_i$ where
  $\sigma_i$~are ordered in non-increasing order and $R^T R = A_{22}$.
\item Choose the cut-off index as $K_c = \max \{ \; i \; | \sigma_i \geq \|
  \vec{\alpha} \|^{-1} tol \; \}$. Construct the method matrix using vectors
  $\vec{u}_1,\ldots,\vec{u}_{K_c}$ from step 3 as
\begin{equation*}
\tilde Q = \begin{bmatrix} I & 0 \\ 0 & \tilde Q_{22} \end{bmatrix}
\quad \text{where} \quad  \tilde Q_{22} = R^{-1} \begin{bmatrix} \vec{u}_1 & \ldots & \vec{u}_{K_c} \end{bmatrix}.
\end{equation*}
\item Solve the eigenvalue problem $\tilde Q^T A \tilde Q \tilde{\vec{x}} = \tilde{\lambda}
  \tilde Q^T M \tilde Q \tilde{\vec{x}}$, e.g., using the Lanczos iteration.
\end{enumerate}

In Step~2, one has to determine a tolerance for computing the exterior eigenvectors $\vec{v}_1,\ldots \vec{v}_{K}$. We proceed to analyse the effect of exterior eigenvector error to the accuracy of the eigenvalues computed using the CPI method. Our analysis relies on perturbation argument identical to one used in Lemma~\ref{DimRedLemma}. When the exterior eigenvectors are incorrectly computed, the space $V_2$ is replaced by $\hat{V}_2$ defined as
\begin{equation}
\label{eq:everrsub}
\hat{V}_2 = \mathop{range}(\hat{P}_{\tilde{\Lambda}}) \oplus \hat{W}_2 \quad \mbox{where} \quad \hat{W}_2 = \mathop{span}_{i=1,\ldots,N} \{ \hat{f}_{\tilde{\Lambda}}(\xi_i) \begin{bmatrix}
M_{21} & A_{21}
\end{bmatrix} \}.
\end{equation}
The matrix $\hat{P}_{\tilde{\Lambda}}$ is the $A_{22}$-orthogonal projection onto $\mathop{span}_{k=1,\ldots,K} \{ \hat{\vec{v}}_k \}$ and $\hat{f}_{\tilde{\Lambda}}(\xi) = (I-\hat{P}_{\tilde{\Lambda}}) (A_{22}-\xi M_{22})^{-1}$. Note that $P_{\tilde{\Lambda}}$ defined in Eq.~(3.3)~is also an $A_{22}$-orthogonal projection. Following \cite{Bo:2010}, the eigenvector error is measured using the gap, i.e., the maximum angle between the exact and the approximate eigenspace in the $A_{22}$-norm as
\begin{equation}
\label{eq:hdist}
d_H := \mathop{max} 
\left(
\max_{ \substack{ \vec u \in \mathop{range}(P_{\tilde{\Lambda}}) \\ \| \vec{u} \|_{A_{22}} = 1}} \| (I-\hat{P}_{\tilde{\Lambda}}) \vec u \|_{A_{22}}, 
\max_{ \substack{ \hat{\vec{u}} \in \mathop{range}(\hat{P}_{\tilde{\Lambda}}) \\ \| \hat{\vec{u}} \|_{A_{22}} = 1}} \| (I-P_{\tilde{\Lambda}}) \hat{\vec{u}} \|_{A_{22}} \right).
\end{equation}
\begin{corollary}
Let $(\lambda,\vec{x}) \in (0,\Lambda) \times \mathbb{R}^{n} \setminus
\{ 0 \}$ be an eigenpair of Eq.~\eqref{eq:AlgEigValP} with
$\vec{x} = \begin{bmatrix} \vec{x}_1 & \vec{x}_2 \end{bmatrix}^T$
according to the standard splitting. By $\hat{\vec{v}}_1,\ldots, \hat{\vec{v}}_{K}$ denote the set of approximate exterior eigenvectors, and by $\hat{P}_{\tilde{\Lambda}}$ the $A_{22}$-orthogonal projection onto $\mathop{span}_{k=1,\ldots,K} \{ \hat{\vec{v}}_k \}$. Let the method subspace $\hat{V}$~and~$\hat{V}_2$ in 
Eq.~\eqref{eq:everrsub} be related as $V$~and~$V_2$ in Eq.~\eqref{eq:subspace_V}. Then there exists $\hat{\lambda} \in \sigma_{\hat V}(A,M)$ such that
  \begin{equation*}
    \frac{|\lambda - \hat{\lambda}|}{\lambda} \leq 2 C_M C(\lambda) \Lambda (4\gamma)^3 \left( \frac{1}{4(\gamma-1)} \right)^{2 N + 2} 
+ C^\prime \max_{\substack{k=1,\ldots,{n_2} \\ i = 1,\ldots,N}} \left( \frac{\mu_k-\lambda}{\mu_k-\xi_i} \right)^2 d^2_H
\end{equation*}
for any parameter value combination $N, \gamma$ with $C^\prime := 2(1+\Lambda_N)^2 \| \vec{x}_2 \|^2_{A_{22}}$. The constants $C_M$ and $C(\lambda)$ are as in Theorem~\ref{th:FinalEstimate}, $d_H$ is the maximal angle defined by Eq.~\eqref{eq:hdist}, and $\Lambda_N$ is defined in Eq.~\eqref{eq:Lebesgue}.
\end{corollary}

\noindent The contribution of inaccurate exterior eigenvectors to relative error in eigenvalues depends inversely on $\mathop{dist}(\{\xi_i\},\{\mu_k\})$. As Chebyshev interpolation points are not nested and $\mu_k$'s have been determined in Step~2, one may be able to adjust the number of interpolation points $N$ so that $\mathop{dist}(\{\xi_i\},\{\mu_k\})$ increases. 

\begin{proof}  Let $\hat{\vec{u}}_2 = \hat{P}_{\tilde{\Lambda}} \vec{x}_2 +\sum_{i=1}^N \ell_i(\lambda) \hat{f}_{\tilde \Lambda}(\xi_i) Z(\lambda) \vec x_1$ and $\vec{u}_2$ be as in Lemma~\ref{DimRedLemma}. Then
\begin{equation*}
\vec{u}_2 - \hat{\vec{u}}_2 = 
( \hat{P}_{\tilde{\Lambda}} - P_{\tilde{\Lambda}} )\vec{x}_2 + 
\sum_{i=1}^N \ell_i(\lambda)(P_{\tilde \Lambda}-\hat{P}_{\tilde \Lambda}) (A_{22}-\xi_i M_{22})^{-1} Z(\lambda) \vec{x}_1.
\end{equation*}
Since $\| \hat{P}_{\tilde{\Lambda}} - P_{\tilde{\Lambda}} \|_{A_{22}} = d_H$ by \cite{BSPIT}, we have
\begin{equation*}
\| \vec{u}_2 - \hat{\vec{u}_2} \|_{A_{22}} 
\leq  d_H \| \vec{x}_2 \|_{A_{22}} 
+  d_H \sum_{i=1}^N |\ell_i(\lambda)| 
 \| (A_{22}-\xi_i M_{22})^{-1} Z(\lambda) \vec{x}_1 \|_{A_{22}}
\end{equation*}
and further,
\begin{equation*}
\| \vec{u}_2 - \hat{\vec{u}_2} \|_{A_{22}} 
\leq 
 d_H \| \vec{x}_2 \|_{A_{22}} 
+ d_H \Lambda_N \max_{i=1,\ldots,N} 
 \| (A_{22}-\xi_i M_{22})^{-1} Z(\lambda) \vec{x}_1 \|_{A_{22}}.
\end{equation*}
It now follows from a similar argument that was used to derive Eq.~\eqref{BetaEstimates} that
\begin{equation*}
 \| (A_{22}-\xi_i M_{22})^{-1} Z(\lambda) \vec{x}_1 \|^2_{A_{22}} 
 = \sum_{k=1}^{n_2} \mu_k \frac{ \beta^2_k(\lambda)}{ (\mu_k-\xi_i)^2 }
 = \max_{k=1,\ldots,{n_2}} \left( \frac{\mu_k-\lambda}{\mu_k-\xi_i} \right)^2 \| \vec{x}_2 \|^2_{A_{22}}.
\end{equation*}
The claim follows by a perturbation argument as in Lemma~\ref{DimRedLemma}.
%
\end{proof}
\begin{remark} An alternative approach for  Step~3~is to directly solve for
  $\tilde{\vec{q}}_{ij} := (I-P_{\tilde\Lambda}) \vec{q}_{ij}$ using the saddle point formulation
\begin{equation*}
\begin{bmatrix} (A_{22} -\xi_i M_{22}) & M_{22} B_1 \\ B_1^T M_{22} & 0 \end{bmatrix} 
\begin{bmatrix} \tilde{\vec{q}}_{ij} \\ \vec{\nu} \end{bmatrix} = 
\begin{bmatrix} \vec{p}_{j} \\ 0 \end{bmatrix}
\end{equation*}
where $B_1 = \begin{bmatrix} \vec{v}_1 & \ldots & \vec{v}_K \end{bmatrix}$ and
$\vec{v}_k$'s are computed in Step~2. This formulation preserves most of the
sparse structure of the linear systems~\eqref{SampleVectors} and is numerically stable when an interpolation point $\xi_i$ is close to $\mu_k$. All numerical experiments were performed without paying attention to this issue.
\end{remark}

Due to memory constraints, it is not always feasible to store vectors $\vec{q}_{ij}$ in
Step~3 or to explicitly construct $B \in \mathbb{R}^{n_2 \times (K+Nr)}$ in
Step~4 when $n_2$ or $K+Nr$ is prohibitively large. In
the construction of the method matrix, Step~5, only vectors
$\vec{u}_1,\ldots,\vec{u}_{K_c}$ corresponding to the largest singular values of
$RB$ are needed. Steps~3 and 4 can be combined into an iterative
solution of the largest singular values of $RB$ and the corresponding vectors
$\vec{u}_i$ using action of $(R B) (RB)^T$ without storing $\vec{q}_{ij}$.

\subsection{2D Rectangle}
\label{sec:rectangle}
We consider numerical solution of Eq.~\eqref{eq:cont_eigen_again} in the
rectangular domain shown in Fig.~\ref{fig:vt_mri}. The homogeneous Dirichlet
boundary condition is used. The domain is uniformly discretised with $67\,872$
triangular elements and $34\,241$ nodes. With the boundary conditions, this
resulted in $n = 33\,633$ degrees of freedom with $n_\Gamma = 163$, $n_1 =
14\,638$, and $n_2 = 18\,995$ in Eq.~\eqref{eq:dec_sys}. The spectral interval
of interest $(0,\Lambda)$ with $\Lambda = 135$ allows us to compute $15$ of the
lowest eigenvalues $\lambda_1,\lambda_2,\ldots$ in non-decreasing order. The
numerically obtained largest relative eigenvalue errors (without using the SVD-based
dimension reduction process of Section~\ref{sec:dim_red}) and its upper bound
from Eq.~\eqref{FinalRelativeErrorEstimate} are shown in
Fig.~\ref{fig:errors_fixed}.

\begin{figure}[h]
  \centering
  \includegraphics[width=.45\textwidth]{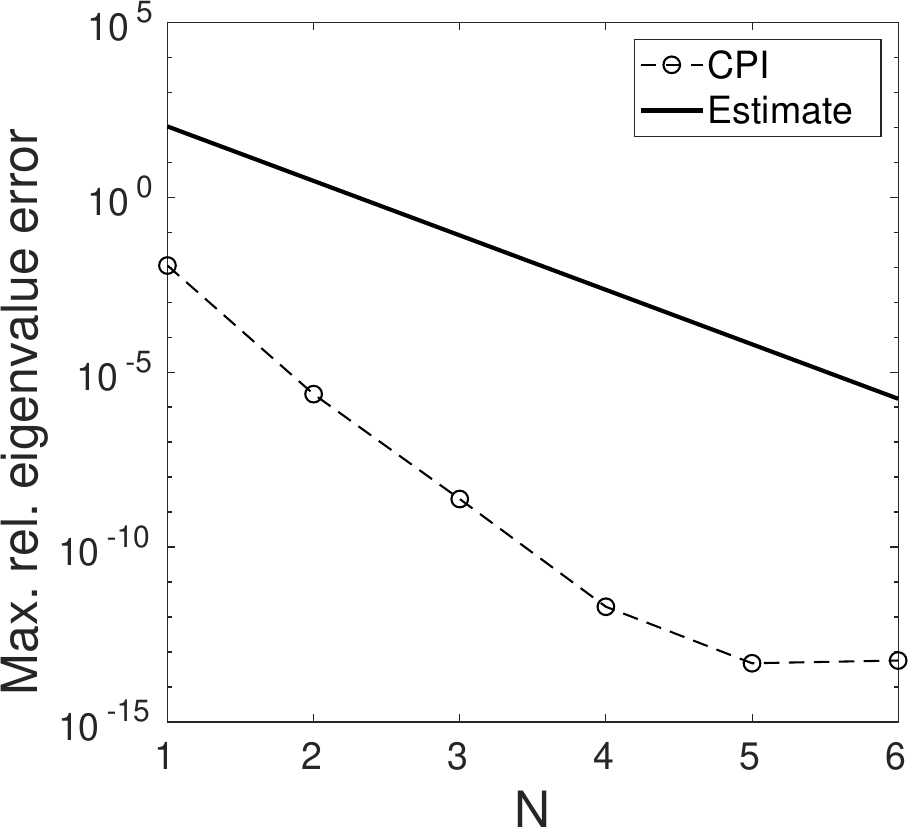} \hspace{18pt}
  \includegraphics[width=.45\textwidth]{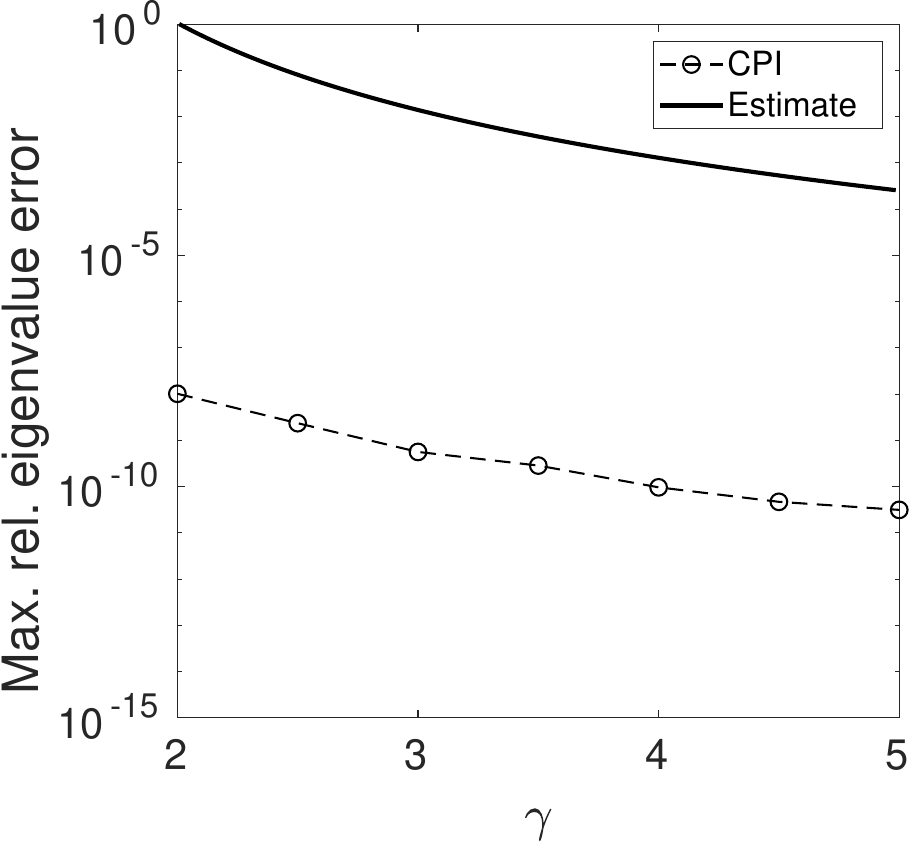}
  \caption{Numerical and theoretical maximal relative errors of eigenvalues
    $\lambda \in (0,135)$ for $\gamma = 2.5$ (left) and $N=3$ (right).}
  \label{fig:errors_fixed}
\end{figure}

The largest relative eigenvalue errors with several values of $\gamma$ and $N$
are compared to the theoretical estimate
(Eq.~\eqref{FinalRelativeErrorEstimate}) in Fig.~\ref{fig:contour_error}. For
validation of the cost model in Eq.~\eqref{eq:cost}, the computational time to
solve Eq.~\eqref{eq:subspace_eigenvalue} with several $N$ and $\gamma$ is
illustrated in Fig.~\ref{fig:contour_time}.

\begin{figure}[h]
  \centering
  \includegraphics[width=.45\textwidth]{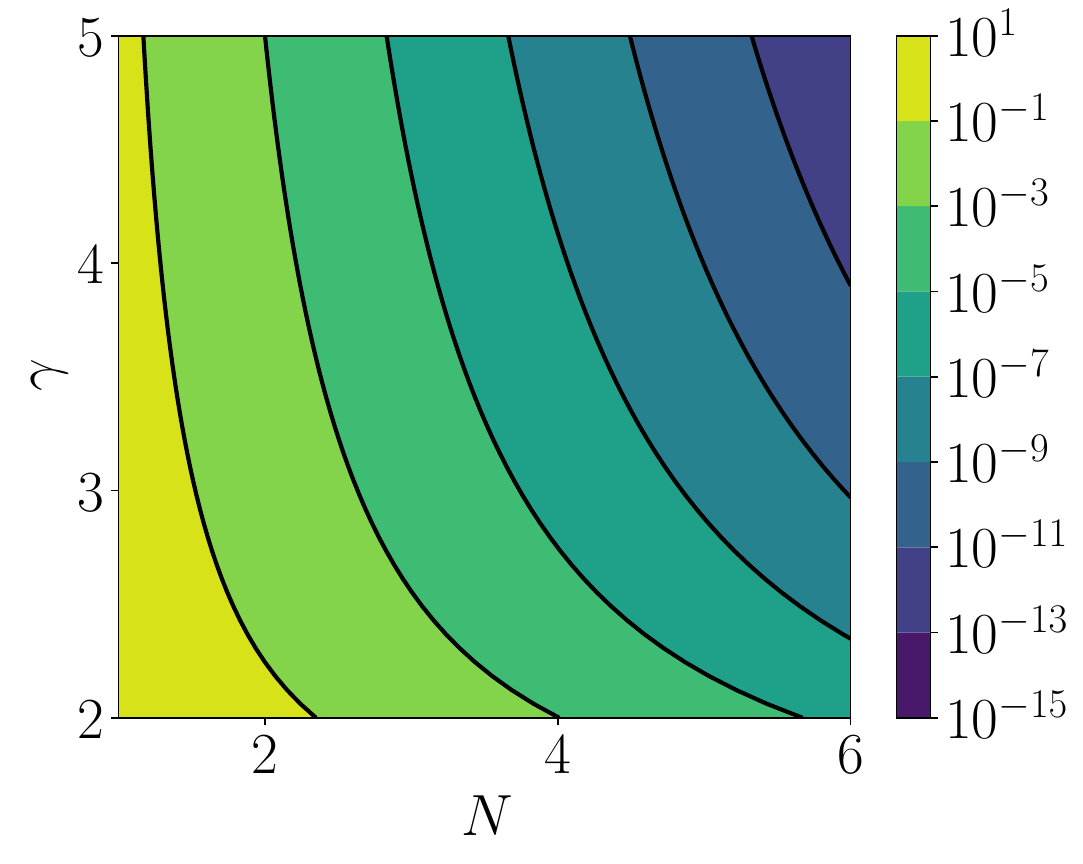}
  \includegraphics[width=.45\textwidth]{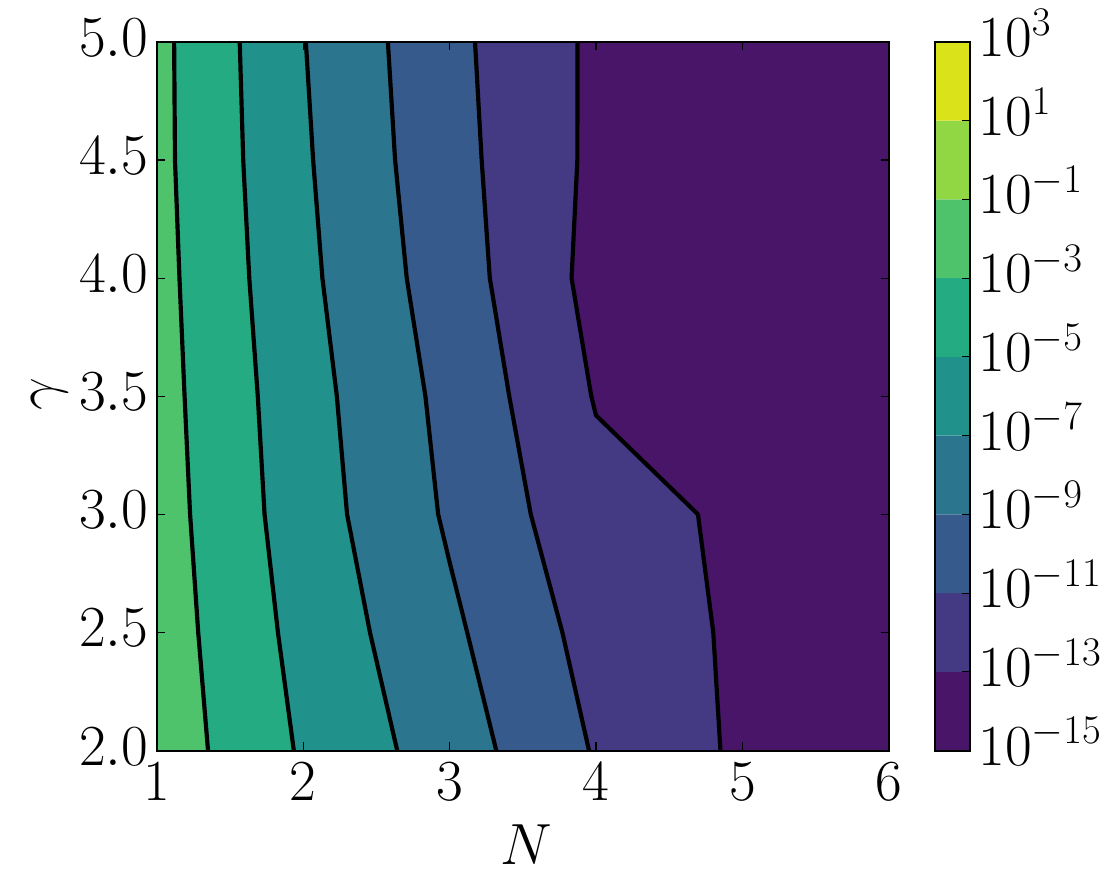}
  \caption{Largest relative eigenvalue errors for $15$ lowest eigenvalues of the
    $2D$ rectangle using several values of $N$ and $\gamma$ from the theoretical
    estimate Eq.~\eqref{FinalRelativeErrorEstimate} (left) and from numerical
    experiments (right). The value of $K(\gamma \Lambda)$ varies from $K=18$
    when $\gamma=2$ to $K=48$ for $\gamma=5$.}
  \label{fig:contour_error}
\end{figure}

The effect of the SVD-based dimension reduction of Section~\ref{sec:dim_red} is
demonstrated in Fig.~\ref{fig:sigmas}. The relative eigenvalue error and $\mathop{dim}(
\tilde{V}_2)$ are given as a function of the truncation error level $tol > 0$.

\begin{figure}[h]
  \centering
  \includegraphics[width=.6\textwidth]{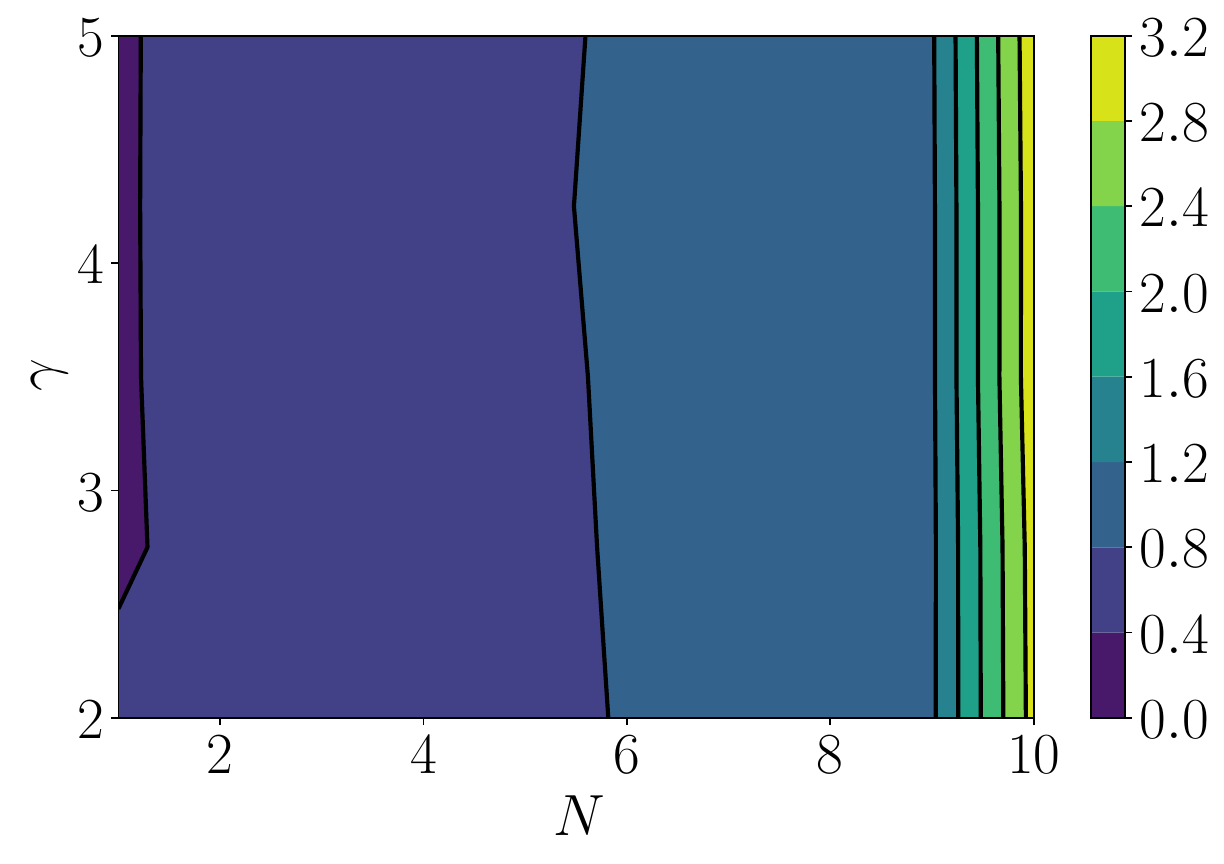}
  \caption{The computation time (in seconds) required by Matlab R2017a \texttt{eigs} to solve the $15$ smallest eigenvalues of
    Eq.~\eqref{eq:subspace_eigenvalue} for several values of $N$ and $\gamma$.
    The time is averaged from $50$ computations. Intel Xeon E5-1630 CPU with 32GB of RAM was used.}
\label{fig:contour_time}
\end{figure}

\begin{figure}[h]
  \centering
  \includegraphics[width=.45\textwidth]{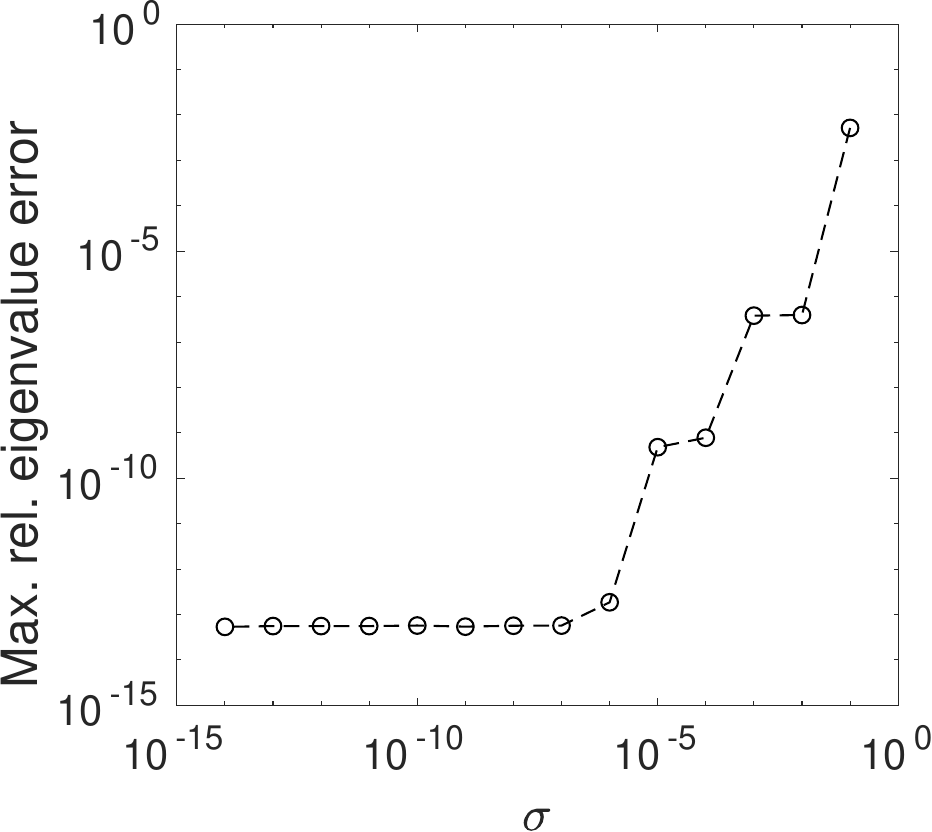} \hspace{18pt}
  \includegraphics[width=.45\textwidth]{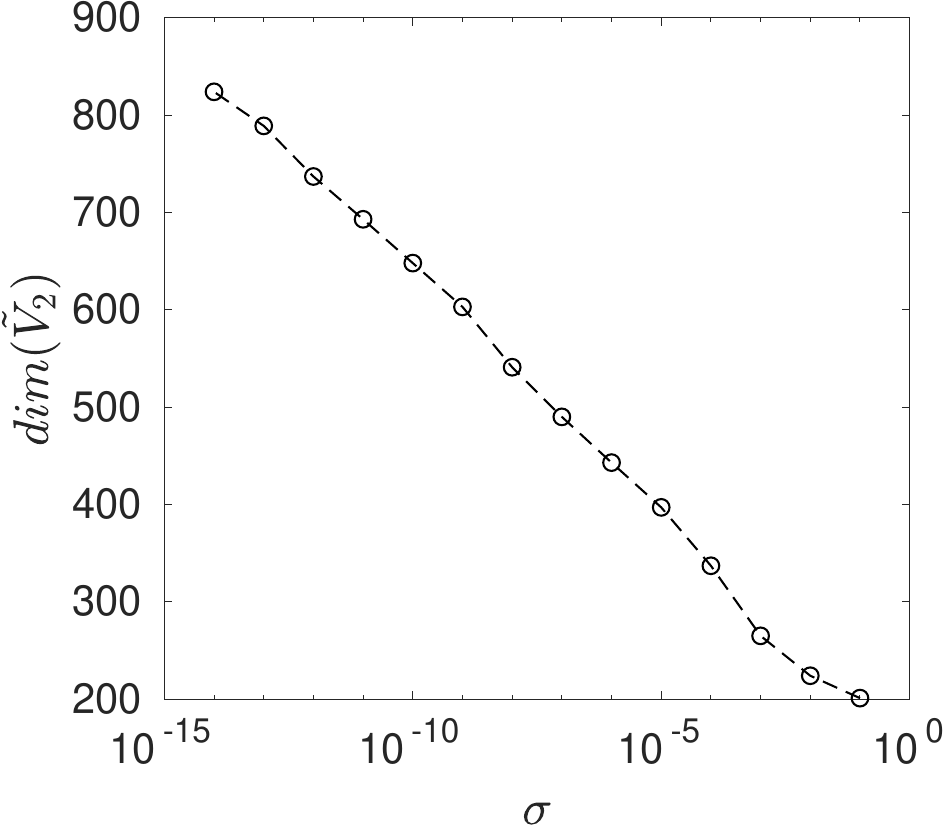}
  \caption{The effect of the cut-off threshold $\sigma$ in SVD to the relative
    error and the size of the method matrix $Q$. The parameter values $\gamma =
    4$ and $N = 6$ were used.}
  \label{fig:sigmas}
\end{figure}

Finally, the CPI method was benchmarked against the CMS implementation detailed
in Section~\ref{sec:cms}. Values for $N$ and $\gamma$ were chosen from
Eqs.~\eqref{eq:gamma2N}~and~\eqref{eq:tol2gamma} given a list of target
error levels. The value of $tol$ was adjusted for each target
error level. The comparison is illustrated in Fig.~\ref{fig:error_cms}.

\subsection{3D Acoustic example}
\label{sec:acoustic}
The computational domain shown in Fig.~\ref{fig:vt_mri} consists of a human
vocal tract geometry $\Omega_1$ and a mock up model of the MRI head coil $\Omega_2$. The vocal tract
geometries were automatically extracted from MRI data as explained
in~\cite{Aalto:MesTec:2014,Ojalammi:Segmentation:2017}, and the interface $\Gamma$ was attached. The vocal tract was embedded into a head model purchased
from Turbosquid~\cite{Turbosquid:2005}.

Homogeneous Dirichlet boundary condition was posed on the areas marked in
Fig.~\ref{fig:vt_bc}, and the Neumann condition was used on other parts of the  boundary.
The interface $\Gamma$ is a spherical surface separating $\Omega_1$ and $\Omega_2$. We use three versions of the 
vocal tract geometry corresponding to  
Finnish vowels \textnormal{[\textipa{\textscripta}],[\textipa{i}]}, and \textnormal{[\textipa{u}]} as visualised in Fig.~\ref{fig:vt_domains}. The domain $\Omega_2$ contains $522\,517$
tetrahedral elements and $101\,222$ nodes, and the interface has $n_\Gamma =
950$ degrees of freedom. Having set the boundary conditions, we have $n_2 = 97\,375$. 

\begin{figure}
\begin{minipage}[t]{0.45\textwidth}
  \centering
  \includegraphics[width=\textwidth]{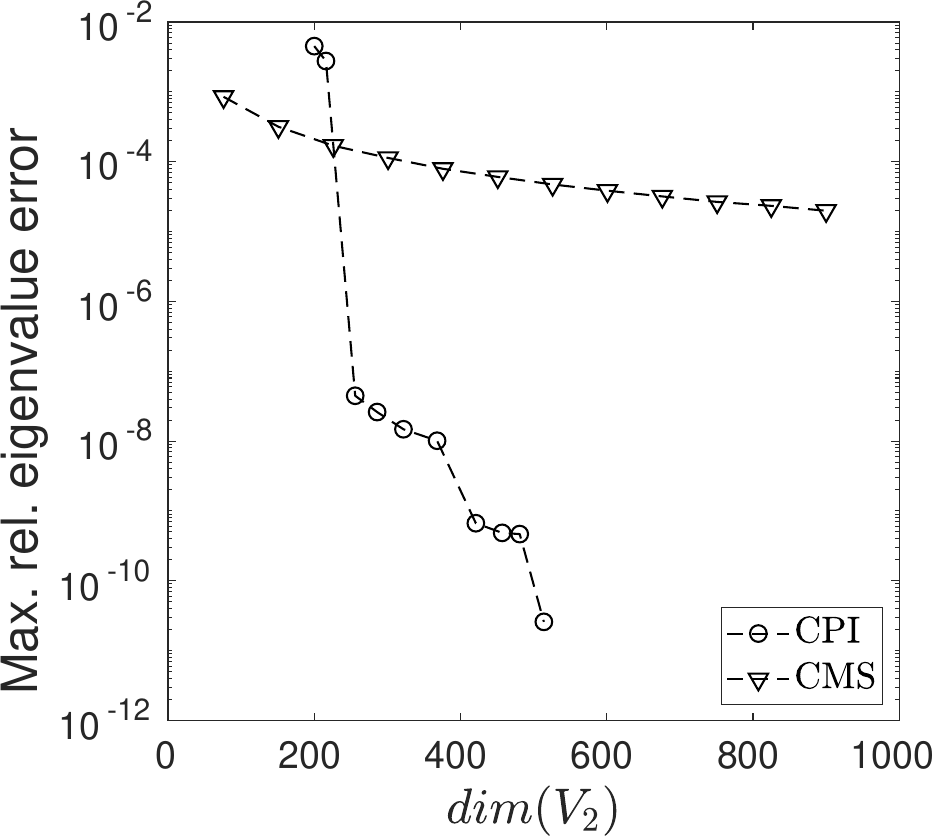}
  \caption{The CPI method compared with CMS in the 2D rectangular
    domain. The $x$-axis corresponds to the amount of degrees of freedom related
    to $\Omega_2$.}
  \label{fig:error_cms}
\end{minipage}
\hfill
\begin{minipage}[t]{0.45\textwidth}
  \centering
  \includegraphics[width=.6\textwidth]{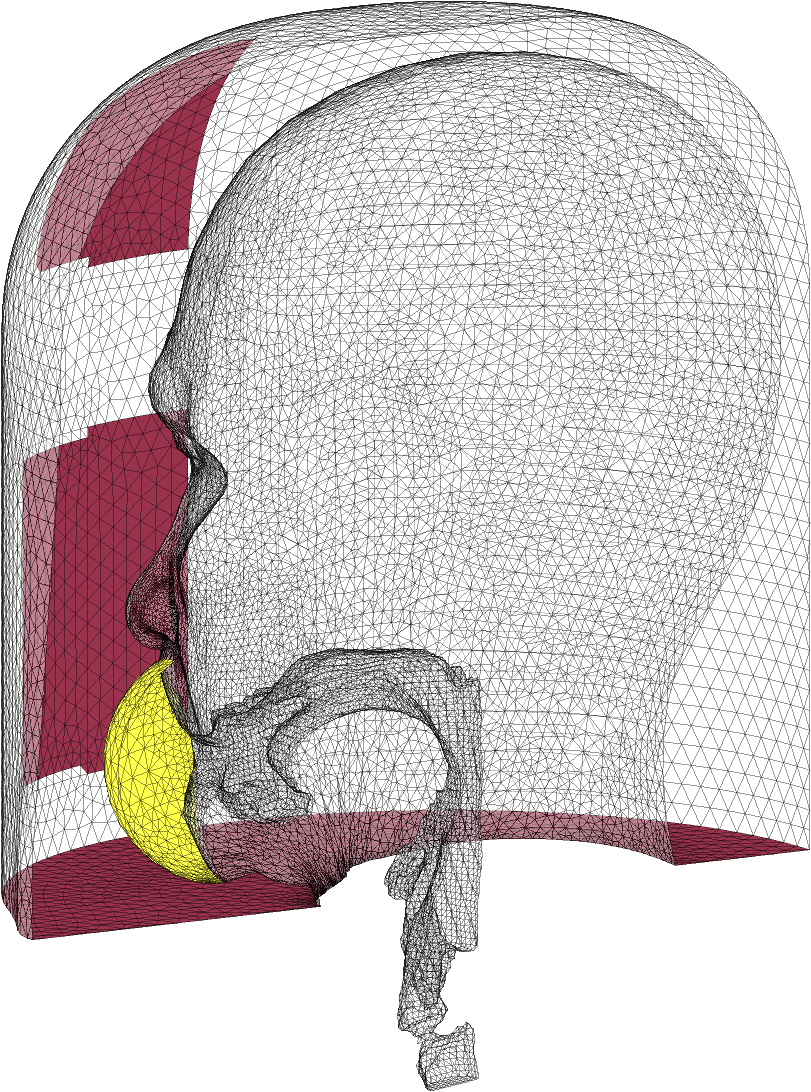}
  \caption{A visualisation of the special degrees of freedom for the acoustic
    example. Dirichlet boundary condition is posed on the red area, and the
    interface $\Gamma$ is marked by the yellow faces.}
  \label{fig:vt_bc}
  \end{minipage}
 \end{figure}

The eigenvalues $\lambda$ in Eq.~\eqref{eq:cont_eigen_again}  and resonant frequencies satisfy $2 \pi f =  c \lambda^{1/2}$ where $c = 344 \tfrac{\mathrm{m}}{\mathrm{s}}$ is the speed of sound at temperature $294$K. Setting the target relative eigenvalue error to $10^{-6}$ corresponds 
to relative error of $10^{-3}$ for resonant frequencies. The spectral interval
of interest $\lambda \in (0,3\,000)$ matches frequencies up to $3\,
\mathrm{kHz}$. Eqs.~\eqref{eq:gamma2N}~and~\eqref{eq:tol2gamma} give $N = 3$,
$\gamma = 8$ for a relative eigenvalue error level $10^{-6}$.

Writing  
\begin{equation*}
  \begin{bmatrix}
    A_{11}  &  A_{12} \tilde{Q}_{22} \\
    A_{21}\tilde{Q}_{22} & \tilde{Q}_{22}^T A_{22} \tilde{Q}_{22}
  \end{bmatrix}
  \tilde{\vec{x}} = \tilde{\lambda}
  \begin{bmatrix}
    M_{11}  &  M_{12} \tilde{Q}_{22} \\
    M_{21}\tilde{Q}_{21} & \tilde{Q}_{22}^T M_{22} \tilde{Q}_{22}    
  \end{bmatrix}
  \tilde{\vec{x}}
\end{equation*}
for a given a method matrix $\tilde{Q}$, we can use the same blocks
$\tilde{Q}_{22}^T A_{22} \tilde{Q}_{22}$ and $\tilde{Q}_{22}^T M_{22}
\tilde{Q}_{22}$ for different versions, and only the off-diagonal blocks 
have to be updated.

A conservative choice for truncation index $K_c$ was made according to target
relative error $10^{-6}$ resulting into subspace $\tilde{V}_2$ with $\mathop{dim}
(\tilde{V}_2) = 1\,601$. Forty lowest resonant frequencies could be computed to the
given tolerance. The computation times and maximum relative eigenvalue errors
are listed in Table~\ref{tab:acoustic_results}. The sample vectors in Step~3
were solved in parallel for each sample point using different cores of the CPU.
The precomputations needed to form the CPI method matrix took 13 minutes and 15
seconds. As shown in Table~\ref{tab:acoustic_results}, the time difference
between solving the original and the reduced eigenvalue problem is around $20$
seconds. Thus, $40$ different versions of the acoustic eigenvalue problems need
to be solved in order to break even in terms of computational time.

\begin{figure}[h]
  \centering
  \includegraphics[width=.3\textwidth]{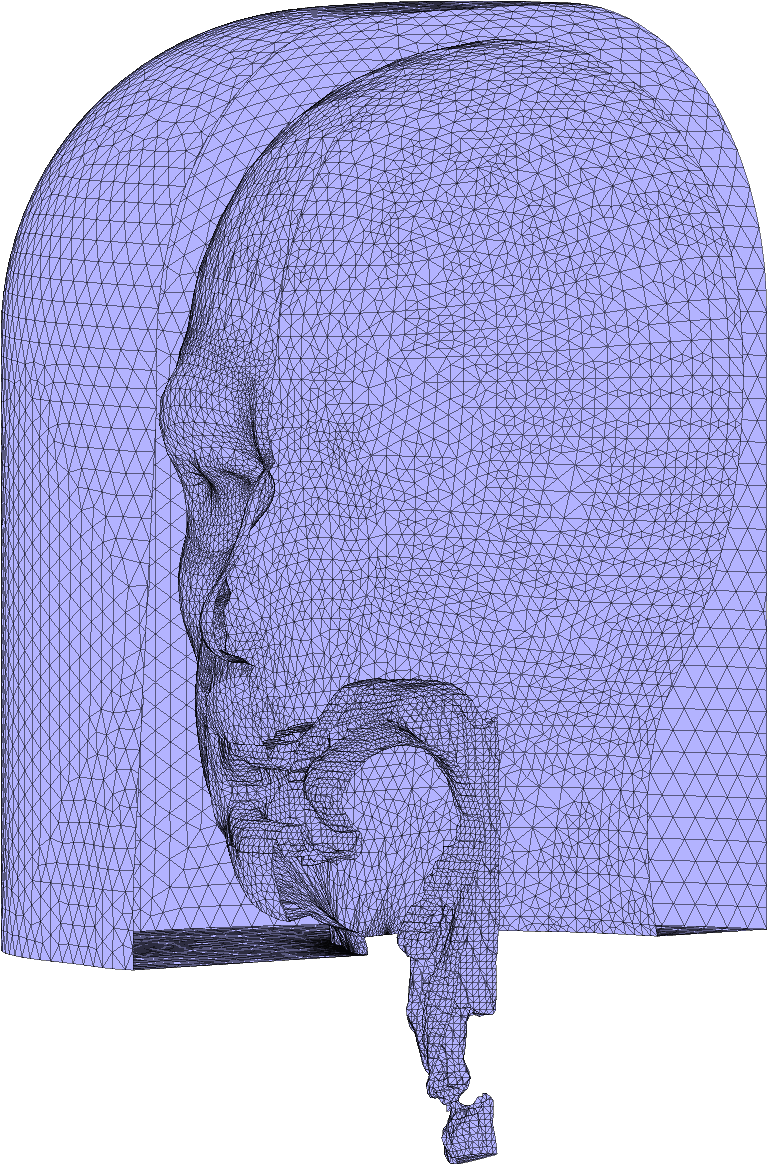} \hfill
  \includegraphics[width=.3\textwidth]{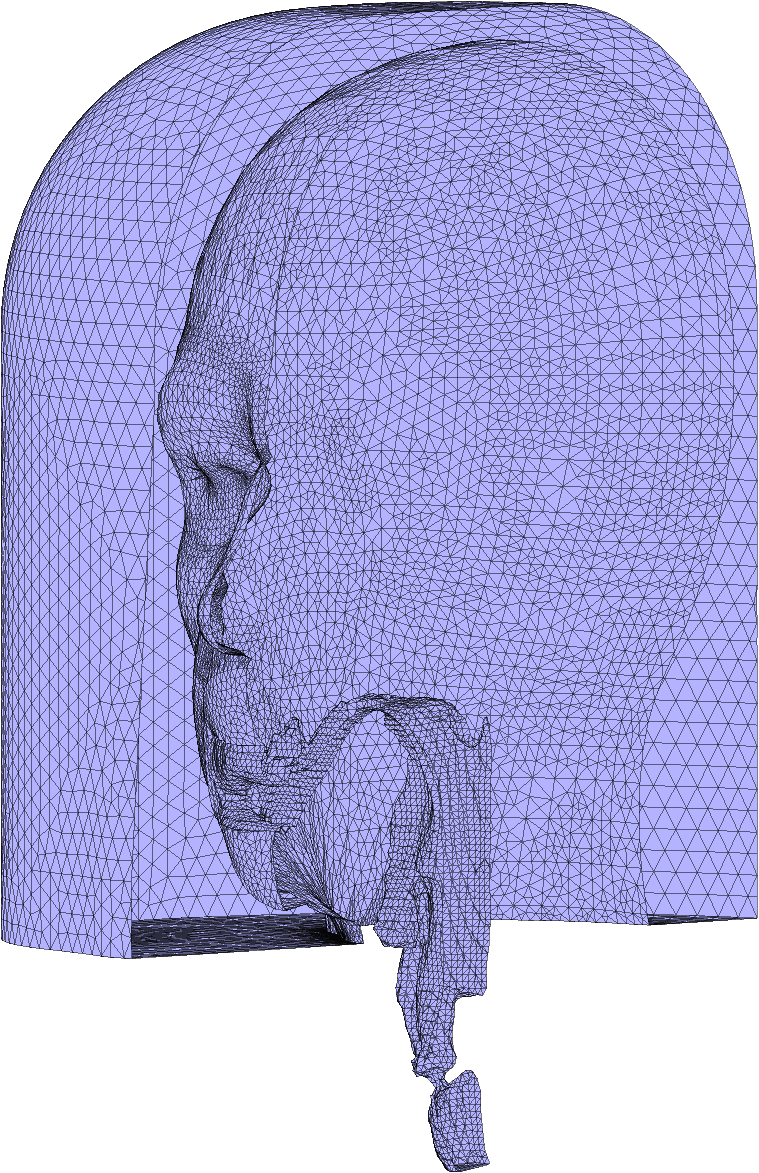} \hfill
  \includegraphics[width=.3\textwidth]{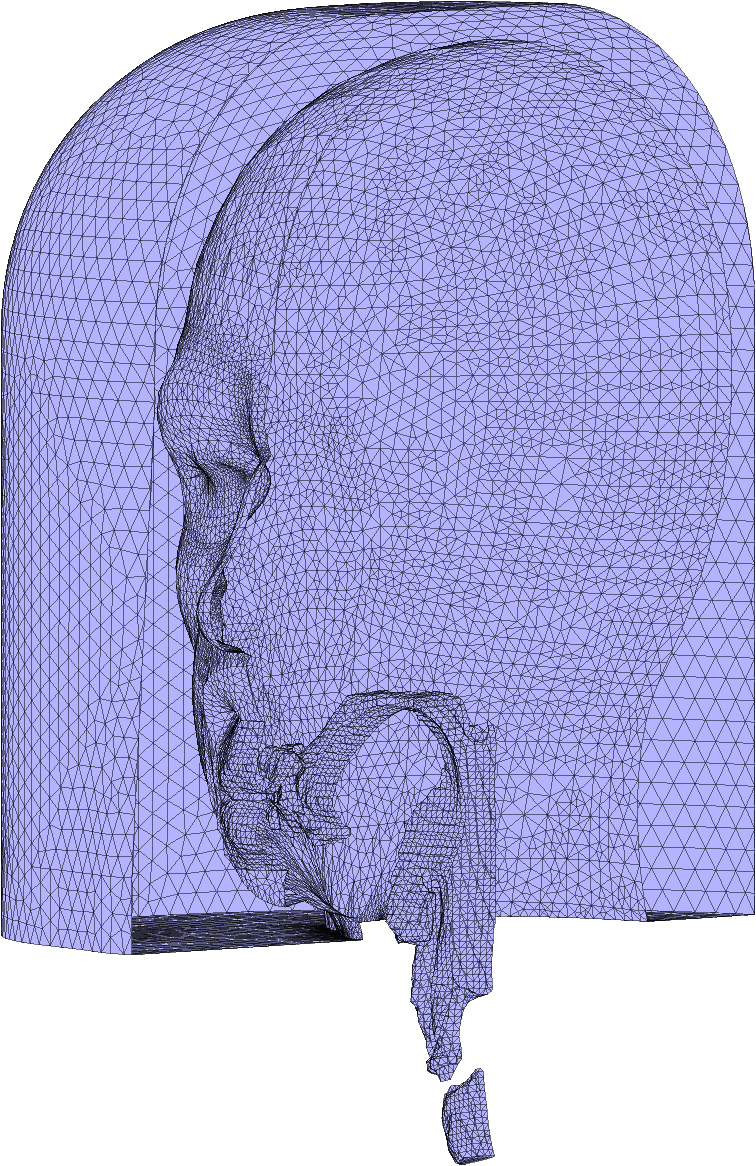}
  \caption{The three vocal tract geometries
    \textnormal{([\textipa{\textscripta}], [\textipa{i}], [\textipa{u}])}
    together with the fixed exterior domain $\Omega_2$. }
  \label{fig:vt_domains}
\end{figure}
\begin{table}
  \centering
  \begin{tabular}{ccccc}
    vowel  & $n_1$  &  CPI & original & max. rel. error    \\
    \hline
    {[\textipa{\textscripta}]}  & $16\,293$    &	$4.44$	&	$26.05$  & $7.05\cdot 10^{-8}$     \\
    {[\textipa{i}]} & $13\,939$ &	$4.57$	&	$28.49$  & $7.28\cdot 10^{-8}$     \\
    {[\textipa{u}]} & $15\,139$ &	$4.66$	&	$25.66$  & $7.36\cdot 10^{-8}$  
  \end{tabular}
  \caption{The computation time (in seconds) required by Matlab  \texttt{eigs} to solve
    the original eigenvalue problem Eq.~\eqref{eq:AlgEigValP} vs. to solve Eq.~\eqref{eq:subspace_eigenvalue} where $\tilde{Q}$ is constructed by dimension reduced CPI. Forty lowest eigenvalues were solved, and their maximum relative
    eigenvalue error is given. The same
    software and hardware were used as specified in
    Fig.~\ref{fig:contour_time}.}
  \label{tab:acoustic_results}
\end{table}

A comparison between CPI and CMS was also performed in the 3D example. The
values $\gamma=8,\; N=3$ were used in CPI with a varying cut-off threshold
$\sigma$ in order to produce comparable dimensions for the subspace
$\tilde{V}_2$. The largest relative eigenvalue error for both methods was
measured using different subspace dimensions. Additionally, the relative
eigenvalue error for each of the $40$ smallest eigenvalues was compared when the
subspace dimension was $1\,600$. For comparison, Matlab \texttt{eigs} uses by
default a subspace of dimension at most $2k$, where $k$ is the number of
eigenvalues to be computed. The results are shown in Fig.~\ref{fig:cpi_vs_cms}.
Computing this many eigenvalues of the exterior system for the CMS method
required $27$ minutes and $41$ seconds, which is approximately twice the time
required by the CPI method. With CPI and the parameters used, $550$ eigenvalues
needed to be used for the exterior system.

\begin{figure}
  \centering
  \includegraphics[width=.45\textwidth]{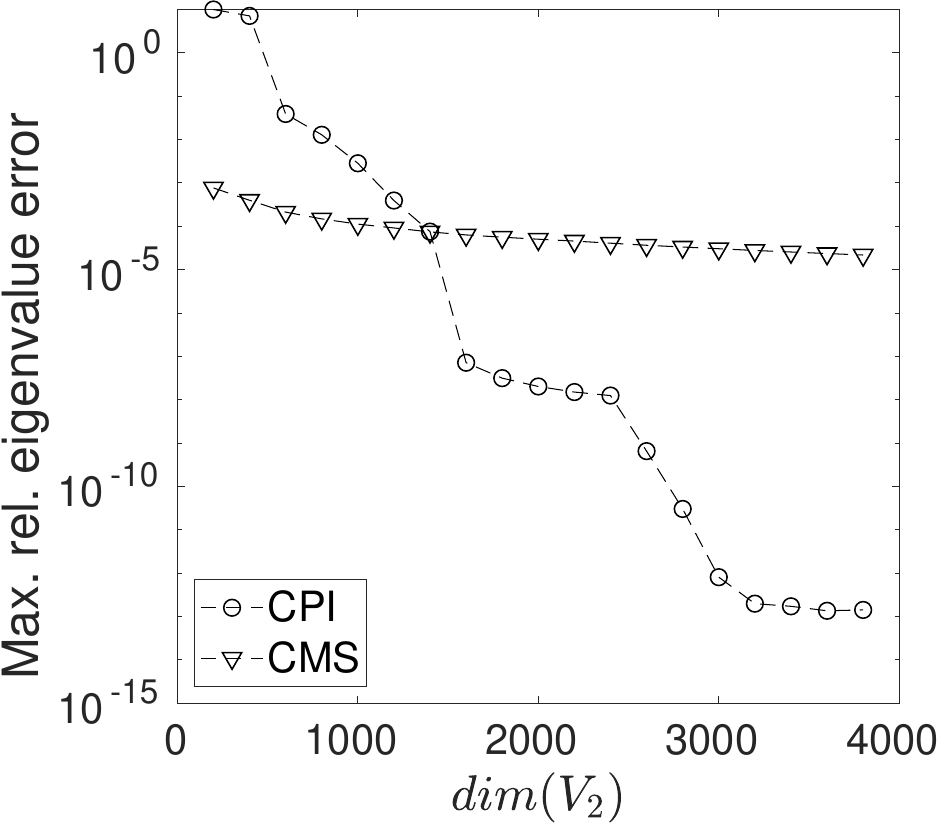} \hfill
  \includegraphics[width=.45\textwidth]{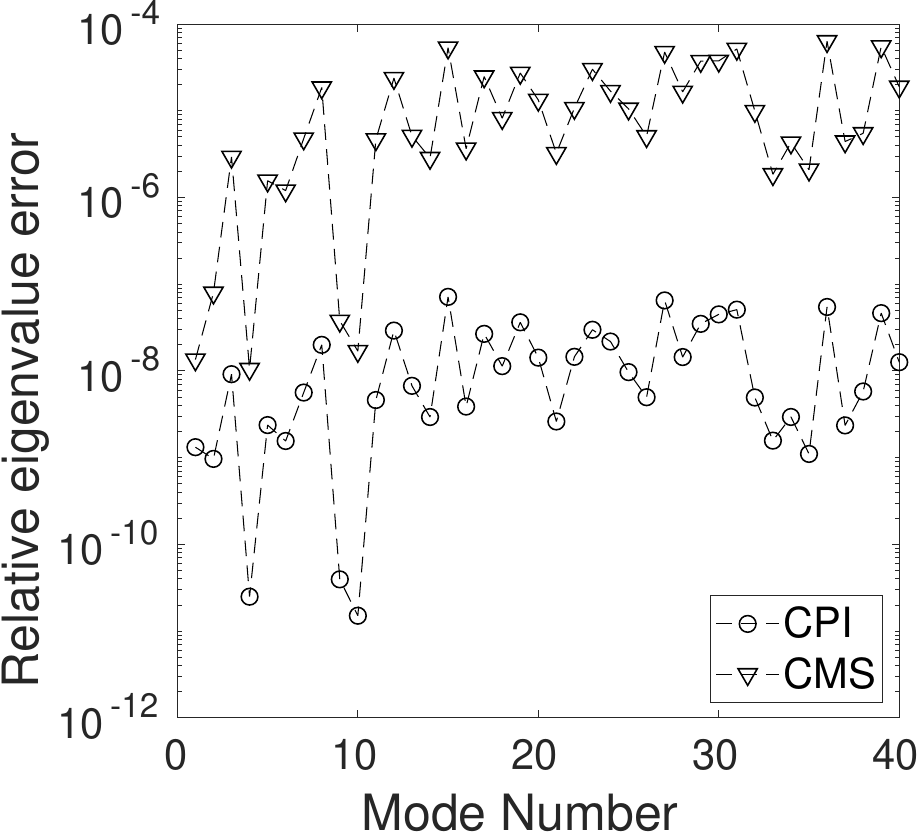}
  \caption{CPI compared to CMS in the 3D acoustic example. Left: largest
    relative eigenvalue error for $40$ of the smallest eigenvalues. Right:
    relative eigenvalue errors for each eigenvalue with $\mathop{dim}(V_2) = 1\,600$ in
    both methods.}
  \label{fig:cpi_vs_cms}
\end{figure}

\section{Conclusions}
We have presented a Condensed Pole Interpolation (CPI) method for efficient
solution of symmetric structured eigenvalue problems by constructing a
particular method subspace $V$. Error analysis for the CPI method shows
convergence of relative eigenvalue error at rate $C \rho^N$ where $N$ is the
number of interpolation points in the spectral interval of interest
$(0,\Lambda)$, and $\rho < 1 $ depends of the oversampling parameter $\gamma >
5/4$. Optimal parameter values for $\gamma$ and $N$ are chosen based on a cost
model. A dimension reduced version of the CPI method with convergence analysis
is given. Numerical experiments on finite element discretised Laplace operator
($d=2,3$) show that the method has practical value and indicate faster
convergence in comparison to the convergence estimate
Eq.~\eqref{FinalRelativeErrorEstimate}. Authors have observed that the CPI
method becomes increasingly competetive against Matlab \texttt{eigs} for large
values of $\Lambda$, excluding precomputations. In addition, the performance is
improved for families of matrices where $n_\Gamma$ and $n_1/n_2$ are small.

The CPI method requires precomputation of a basis for the method subspace $V$.
This involves solution of exterior eigenvectors $(\mu_k,\vec{v}_k)$ satisfying
$\mu_k \leq \gamma \Lambda$ and the solution of $n_\Gamma N$ linear systems.
Additional precomputation cost is produced by SVD in dimension reduction. The
cost model used for optimising $\gamma$ and $N$ does not account for any
precomputation costs. In the acoustic eigenvalue problem used for benchmarking,
the CPI method becomes competitive in comparison to \texttt{eigs} after $40$
eigensolves if precomputation time is taken into account. Even costly
precomputations are justified in processes where eigensolution must be repeated
several times by the designer.

The computational speed of the CPI method comes at a price; the algorithm is
memory intensive compared to \texttt{eigs}. In the acoustic eigenvalue problem
involving the $40$ lowest eigenvalues, the method subspace is of dimension
$1\,601$, whereas the subspace used by \texttt{eigs} is of dimension $80$. The
feasibility of precomputations may be limited by storage more severely than by
computation time, in particular, if dimension reduction is not used.

Application of the CPI method to domain decomposition in the finite element
context presents a topic for future work. In fact, the presented analysis
already extends to the setting where $A_{22}$ and $M_{22}$ are block diagonal
matrices as is the case when multiple subdomains are treated.

\section{Acknowledgements} The geometry for the exterior model in
Section~\ref{sec:acoustic} is loosely based on the MRI head coil design provided
by Siemens Healthineers. The authors are grateful for the comments of the reviewers.

\bibliographystyle{siamplain}
\bibliography{master_new}

\end{document}